\documentclass{article}
\usepackage{amscd,amsmath,amssymb,amsthm,amsfonts,epsfig,graphics}
\usepackage{cite}
\usepackage{longtable}
\usepackage{siunitx}
\usepackage{booktabs} 
\usepackage{tabularx} 
\usepackage{tabu} 
\usepackage{graphics}
\newtheorem{theorem}{Theorem}
\theoremstyle{plain}

\newtheorem{lemma}{Lemma}

\numberwithin{equation}{section}
\numberwithin{lemma}{section}
\numberwithin{theorem}{section}
\numberwithin{corollory}{section}
\numberwithin{proposition}{section}
\begin{document}

		\title{On the Lie nilpotency index of  modular  group algebras}
		\author{
			Suchi Bhatt and Harish Chandra \footnote{\text{CONTACT} hcmsc@mmmut.ac.in, Department of Mathematics and Scientific Computing,  Madan Mohan Malaviya University of Technology, Gorakhpur (U.P.), India-273010.}%
			\\	Department of Mathematics and Scientific Computing\\  Madan Mohan Malaviya University of Technology \\Gorakhpur (U.P.), India
		}
		\date{}			
		\maketitle
		
		\begin{abstract}
			Let $KG$ be the modular group algebra of an arbitrary group $G$  over  a field  $K$  of characteristic $p>0$. It is seen that if $KG$ is Lie nilpotent, then its lower as well as upper Lie nilpotency index is at least  $p+1$. The classification of  group algebras $KG$ with upper Lie nilpotency index $t^{L}(KG)$ upto  $9p-7$ have already been determined. In this paper, we  classify the modular group algebra $KG$ for which the upper Lie nilpotency index  is $10p-8$.
			
		\end{abstract}
		
		\smallskip
		\noindent \text{KEYWORDS}\; Group algebras, Lie nilpotency index
		
		\smallskip
		\noindent \text{2010} MATHEMATICS SUBJECT CLASSIFICATION: 16S34, 17B30

     \section{Introduction}
     Let $KG$ be the group algebra of a group $G$ over a field $K$ of characteristic $p>0$. The group algebra $KG$ can be regarded as a associated Lie algebra  of $KG$, via the Lie commutator $[x,y] = xy-yx$, $\forall x, y\in KG$. Set $[x_{1}, x_{2},...x_{n}] = [[x_{1}, x_{2},...x_{n-1}],x_{n}]$, where  $x_{1}, x_{2},...x_{n}\in KG$. The $n^{th}$ lower Lie power $KG^{[n]}$ of $KG$ is the associated ideal generated by the Lie commutators $[x_{1},x_{2},...x_{n}]$, where $KG^{[1]} = KG$. By induction, the $n^{th}$ upper Lie power $KG^{(n)}$ of $KG$ is the associated ideal generated by all the Lie commutators $[x,y]$, where $x\in KG^{(n-1)}$, $y\in KG$ and $KG^{(1)} = KG$. $KG$ is said to be upper Lie nilpotent (lower Lie nilpotent) if there exists $m$ such that $KG^{(m)} = 0 \, (KG^{[m]} = 0)$. The minimal non-negative integer $m$ such that $KG^{(m)} =0$ and $KG^{[m]} = 0$  is known as the upper Lie nilpotency index and lower Lie nilpotency index of $KG$, denoted by $t^{L}(KG)$ and $t_{L}(KG)$ respectively. It is well known that, if $KG$ is Lie nilpotent, then $p+1\leq t_{L}(KG)\leq |G'|+1$ (see \cite { Shs2, Sharmavisht}).  According to Bhandari and Passi \cite{BP}, if $p>3$ then $t^{L}(KG) = t_{L}(KG)$. But the question when $t_{L}(KG) = t^{L}(KG)$ for $p= 2, 3$ is still open in general. The subgroup $D_{(m),K}(G) = G\cap (1+KG^{(m)}), m\geq 1$ is called the $m^{th}$ Lie dimension subgroup of $G$ and by Passi \cite{passi},  we have \
     $$D_{(m),K}(G) =  \prod_{(i-1)p^{j}\geq m-1}\gamma_{i}(G)^{p^j}.$$
      Let $p^{d_{(m)}} = |D_{(m),K}(G):D_{(m+1),K}(G)|$,  $m\geq 2$. If $KG$ is Lie nilpotent such that $|G'|  = p^{n}$, then according to Jenning's theory \cite{Shalev6}, we have $t^{L}(KG) = 2+(p-1)\Sigma_{m\geq 1} \; md_{(m+1)}$ and $\Sigma_{m\geq 2} \; d_{(m)} = n$. \
     Shalev \cite{Shalev7} initiated the study of group algebras with maximum Lie nilpotency index. This problem was completed by \cite{bovdi3}. Results on the next smaller Lie nilpotency index can be easily seen in  \cite{bovdi2,bovdi3,bovdi4,bovdi5}. In \cite{BK}, Bovdi  and Kurdics  discussed the upper and lower Lie nilpotency index of a modular group algebra of  metabelian group $G$ and determine the nilpotency class of the group of units. Recently, we have some results  on classification of Lie nilpotent group algebras of Lie nilpotency index upto 14 (see \cite{ChS,rsrks,rrm,sbhcms}). Furthermore, group algebras with minimal Lie nilpotency index $p+1$ have been classified by Sharma and Bist \cite{Sharmavisht}. A complete description of the Lie nilpotent group algebras with next possible nilpotency indices $2p$, $3p-1$, $4p-2$, $5p-3$, $6p-4$, $7p-5$, $8p-6$ and $9p-7$ is given in \cite{ms,msbs,msbs2,msbs3,Shalev4}. In this article, we will classify group algebras with upper Lie nilpotency index $10p-8$. For a prime $p$ and  positive integer $x$, $\vartheta_{p'}(x)$ is the maximal divisor of $x$ which is relatively prime to $p$. Also $S(n,m)$ denotes the small group number $m$ of order $n$ from the Small Group Library-Gap \cite{gapg}. We use the following Lemma throughout our work:\
        
     \begin{lemma}(\cite{Shalev7})
     	Let $K$ be a field with $Char K = p> 0 $ and  $G$ be a nilpotent group such that $|G'|= p^{n}$ and $exp(G') = p^{l}$.
     	\begin{enumerate} \label{2l1}
     		
     		\item If $ d_{(l+1)} = 0$ for some $l<pm$, then $d_{(pm+1)}\leq d_{(m+1)}$.
     		\item If $d_{(m+1)} = 0$, then $d_{(s+1)} = 0$ for all $ s\geq m$ with $\vartheta_{p'}(s)\geq \vartheta_{p'}(m)$ where $\vartheta_{p'}(x)$ is the maximal divisor of $x$ which is relatively prime to $p$.
     	\end{enumerate}
     	
     \end{lemma}
 \begin{theorem}
 	Let $G$ be a group and $K$ be a field of characteristics $p>0$ such that $KG$ is Lie nilpotent. Then $t^{L}(KG) = 10p-8$ if and only if one of the following condition  satisfied:
 	\begin{enumerate}
 		
 	\item  $G'\cong C_{7^{2}}\times (C_{7})^{2}$ and $\gamma_{3}(G) \subseteq G'^{7}$; 
 	\item $ G'\cong C_{7^{2}}\times C_{7}$, $\gamma_{3}(G)\cong C_{7}$  and $|\gamma_{3}(G)\cap G'^{7}| = 1$;
 	
 	\item  $G'\cong C_{7^{2}}\times C_{7}$, $\gamma_{4}(G) \subseteq G'^{7}\subseteq \gamma_{3}(G)\cong (C_{7})^{2}$ and $\gamma_{5}(G) = 1$;
 	
 	\item $G'\cong C_{5^{2}}\times  (C_{5})^{4} $, $G'^{5}\subseteq \gamma_{3}(G)$ and $\gamma_{4}(G) = 1$;
 	\item $G'\cong (C_{5})^{6}$, $|G'^{5}\cap \gamma_{3}(G)| =1$, $\gamma_{3}(G)\cong C_{5}$ and $\gamma_{4}(G) = 1$;
 	\item $G'\cong  (C_{5^2})^{2}  \times C_{5}$ and  $\gamma_{3}(G)\subseteq G'^{2}$;
 	\item $G'\cong C_{5^2}\times (C_{5})^{3}$, either $|G'^{5}\cap \gamma_{3}(G)| = 1$, $\gamma_{3}(G)\cong C_{5}$  or $G'^{5}\subseteq \gamma_{3}(G)\cong (C_{5})^{2}$;
 	\item $G'\cong (C_{5})^{5}$,  $|G'^{5}\cap \gamma_{3}(G)| = 1$, $\gamma_{3}(G)\cong (C_{5})^{2}$ and $\gamma_{4}(G)= 1$;
 	
 	\item $G'$ is  one of the groups  $S(3125,2)$, $S(3125,40)$, $S(3125,41)$, $S(3125,42)$, $S(3125,43)$, $S(3125,44)$, $S(3125,73)$ or $S(3125,74)$, $G'^{5}\subseteq \zeta(G')$, $G''\subseteq \zeta(G')$, $G'^{5}\subseteq \gamma_{3}(G)\cong (C_{5})^{2}$, $\gamma_{4}(G) \cong C_{5}$ and  $\gamma_{5}(G)= 1$;
 	\item  $G'\cong  C_{5^2}\times (C_{5})^{2}$, either $|G'^{5}\cap \gamma_{3}(G)| = 1$, $\gamma_{3}(G)\cong (C_{5})^{2}$ or $G'^{5}=  \gamma_{3}(G)\cong C_{5}$ or $G'^{5} \subseteq  \gamma_{3}(G)\cong (C_{5})^{3}$;
 	\item $G'\cong C_{8}\times (C_{2})^{3}$, $\gamma_{3}(G)\subseteq G'^{2}$, $\gamma_{3}(G)\cong C_{4}$ and $\gamma_{4}(G)= 1$;
	\item $G'\cong (C_{4})^{2} \times (C_{2})^{2}$, $\gamma_{3}(G)\subseteq G'^{2}$ and  $\gamma_{4}(G)= 1$;
 	\item  $G'\cong C_{4}\times (C_{2})^{4}$, $G'^{2}\subseteq \gamma_{3}(G)\cong C_{4}$ and $\gamma_{4}(G) =1$;
    \item $G'$ is   one of the groups  $S(64,199)$ to $S(64,201)$ or $S(64,215)$ to $S(64,245)$, $\gamma_{3}(G)\subseteq G'^{2}$ and $\gamma_{4}(G)\cong C_{2}$;
    \item $G'$ is  one of the groups  $S(64,264)$ or $S(64,265)$, either $G'^{2}\subseteq \gamma_{3}(G)\cong C_{4}$ or $|G'^{2}\cap \gamma_{3}(G)| = 1$, $\gamma_{3}(G)\cong C_{2}$;
    \item $G'$ is  one of the groups  $S(64,247)$ or $S(64,248)$, $G'^{2} = \gamma_{3}(G) \cong C_{4}$, $\gamma_{4}(G)\cong C_{2}$ and $\gamma_{5}(G) = 1$;
    \item $G' \cong S(64,263)$, $|G'^{2}\cap \gamma_{3}(G)| = 2$, $\gamma_{3}(G)\cong C_{4}$, $\gamma_{4}(G)\cong C_{2}$ and $\gamma_{5}(G) = 1$;
    
 	\item $G'\cong C_{8}\times C_{4}$, either $G'^{2}\subseteq \gamma_{3}(G) \cong C_{4}\times C_{2}$ or $\gamma_{3}(G)\subseteq G'^{2}$, $\gamma_{3}(G)\cong C_{4}$;
 	\item  $G'\cong C_{8}\times (C_{2})^{2}$  and $G'^{2}\subseteq \gamma_{3}(G) \cong C_{4}\times C_{2}$;
 	\item $G'\cong (C_{4})^{2} \times C_{2}$, either $|G'^{2}\cap \gamma_{3}(G)| = 2$, $\gamma_{3}(G) \cong C_{4}$ or $|G'^{2}\cap \gamma_{3}(G)| = 4$, $\gamma_{3}(G) \cong C_{4}\times C_{2}$;
 	\item $G'\cong C_{4}\times (C_{2})^{3}$, $|G'^{2}\cap \gamma_{3}(G)| = 2$ and $\gamma_{3}(G)\cong C_{4}\times C_{2}$;
 	\item $G'$ is   one of the groups  $S(32,4)$, $S(32,5)$ or $S(32,12)$, $\gamma_{3}(G)\subseteq G'^{2}\cong C_{4}\times C_{2}$, $\gamma_{4}(G)\subseteq G'^{4}\gamma_{3}(G)^{2} \cong C_{2}$ and $\gamma_{5}(G) = 1$;
 	\item $G'$ is   one of the groups  $S(32,22)$ to $S(32,26)$,  either $|G'^{2}\cap \gamma_{3}(G)| = 2$, $\gamma_{3}(G)\cong C_{4}$ or $G'^{2}\subseteq \gamma_{3}(G) \cong C_{4}\times C_{2}$, $\gamma_{4}(G)\subseteq G'^{4}\gamma_{3}(G)^{2}\cong C_{2}$, $\gamma_{5}(G) = 1$;
 	\item $G'$ is   one of the groups  $S(32,37)$ or $S(32,38)$, either $|G'^{2}\cap \gamma_{3}(G)| = 2$, $\gamma_{3}(G)\cong (C_{2})^{2}$ or $G'^{2}\subseteq \gamma_{3}(G)\cong C_{4}\times C_{2}$, $\gamma_{4}(G)\subseteq G'^{4}\gamma_{3}(G)^{2}\cong C_{2}$, $\gamma_{5}(G) = 1$;
 	\item $G'$ is   one of the groups  $S(32,46)$ to $S(32,48)$, either $|G'^{2}\cap \gamma_{3}(G)| = 1$, $\gamma_{3}(G) \cong C_{4}$ or $G'^{2}\subseteq \gamma_{3}(G)\cong C_{4}\times C_{2}$, $\gamma_{4}(G)\subseteq G'^{4}\gamma_{3}(G)^{2}\cong C_{2}$, $\gamma_{5}(G) = 1$;
 	
 	\item  $G'\cong (C_{p})^{4}$,   $|G'^{p}\cap \gamma_{3}(G)| = 1$, $\gamma_{3}(G)\cong (C_{p})^{3}$,  $\gamma_{4}(G)\cong (C_{p})^{2}$ and $\gamma_{5}(G) \cong C_{p}$ for $p\geq 5$;
 	\item  $G'\cong C_{9}\times (C_{3})^{2}$, either $|G'^{2}\cap \gamma_{3}(G)| = 1$, $\gamma_{3}(G)\cong (C_{3})^{2}$ or $G'^{3}\subseteq \gamma_{3}(G) \cong (C_{3})^{3}$;
 	\item $G'\cong (C_{3})^{4}$,   $|G'^{3}\cap \gamma_{3}(G)| = 1$ and $\gamma_{3}(G)\cong (C_{3})^{3}$;
 	\item  $G'\cong  C_{8}\times C_{2}$,  either $\gamma_{3}(G)\cong C_{2}$, $|G'^{2}\cap \gamma_{3}(G)| = 1$ or $G'^{2}\subseteq \gamma_{3}(G) \cong C_{4}\times C_{2}$;
 	\item $G'\cong C_{4}\times (C_{2})^{2}$ and  $G'^{2}\subseteq \gamma_{3}(G)\cong C_{4}\times C_{2}$;
 	\item  $G'\cong  ((C_{p}\times C_{p})\rtimes C_{p})\times C_{p}$, $\gamma_{3}(G)\cong (C_{p})^{3}$, $\gamma_{4}(G)\cong (C_{p})^{2}$, $\gamma_{5}(G)\cong C_{p}$ and $\gamma_{6}(G) = 1$ for $p\geq5$;
 	\item $G'\cong (C_{9})^{2}$ and $G'^{3}\subseteq \gamma_{3}(G) \cong (C_{3})^{3}$;
 	\item $G'\cong C_{9}\times (C_{3})^{2}$,  either $|G'^{3}\cap \gamma_{3}(G)| = 1$, $\gamma_{3}(G) \cong (C_{3})^{2}$ or $G'^{3}\subseteq \gamma_{3}(G) \cong (C_{3})^{3} $;
 	\item  $G'\cong  (C_{3})^{4}$, $|G'^{3}\cap \gamma_{3}(G)| = 1$ and $\gamma_{3}(G) \cong (C_{3})^{3} $;
 	\item $G'\cong (C_{p})^{5}$, $\gamma_{3}(G) \cong (C_{p})^{3}$ and  $|G'^{p} \cap \gamma_{3}(G)| = 1$  for $p\geq 5$;
 	\item $G'\cong (C_{9})^{2} \times C_{3}$ and $G'^{3}\subseteq \gamma_{3}(G)\cong (C_{3})^{3}$;
 	\item $G'\cong C_{9}\times (C_{3})^{3}$, either $|G'^{3}\cap \gamma_{3}(G)| = 1$, $\gamma_{3}(G)\cong (C_{3})^{2}$ or $G'^{3}\subseteq  \gamma_{3}(G) \cong (C_{3})^{3}$;
 	\item $G'\cong (C_{3})^{5}$, $|G'^{3}\cap \gamma_{3}(G)| = 1$ and $\gamma_{3}(G) \cong (C_{3})^{3}$;
 	\item  $G'\cong  (C_{4})^{2}\times C_{2}$,  either $G'^{2}\subseteq \gamma_{3}(G)\cong (C_{2})^{3}$ or $|G'^{2}\cap \gamma_{3}(G)| = 2$, $\gamma_{3}(G) \cong (C_{2})^{2} $ or $|G'^{2}\cap \gamma_{3}(G)| = 1$, $\gamma_{3}(G)\cong C_{2}$;
 	\item $G'\cong C_{4}\times (C_{2})^{3}$,  either $G'^{2}\subseteq \gamma_{3}(G)\cong (C_{2})^{3}$ or $|G'^{2}\cap \gamma_{3}(G)| = 1$, $\gamma_{3}(G)\cong (C_{2})^{2} $;
 	\item $G'\cong (C_{2})^{5}$,  $|G'^{2}\cap \gamma_{3}(G)| = 1$ and $\gamma_{3}(G)\cong (C_{2})^{3}$;
 	\item $G'\cong S(32,2)$, $\gamma_{3}(G)\subseteq G'^{2}$, $\gamma_{4}(G)\cong C_{2}$ and $\gamma_{5}(G) = 1$;
 	\item $G'$ is   one of the groups  $S(32,22)$ to $S(32,26)$, $\gamma_{4}(G)\cong C_{2}$, $\gamma_{5}(G) = 1$,  $|G'^{2}\cap \gamma_{3}(G)| = 2$ and  $\gamma_{3}(G)\cong (C_{2})^{2}$;
 	\item $G'$ is  one of the groups  $S(32,46)$ to  $S(32,48)$, $\gamma_{4}(G)\cong C_{2}$, $\gamma_{5}(G) = 1$, either $|G'^{2}\cap \gamma_{3}(G)| = 1$, $\gamma_{3}(G)\cong (C_{2})^{2}$ or $G'^{2}\subseteq \gamma_{3}(G)\cong (C_{3})^{3}$;
 
    \item 	$G'$ is  one of the groups  $S(243,2)$, $S(243,33)$, $S(243,34)$ or $S(243,36)$,  either $G'^{3}\subseteq \gamma_{3}(G)\cong (C_{3})^{3}$ or  $|G'^{3}\cap \gamma_{3}(G)| = 3$, $\gamma_{3}(G)\cong (C_{3})^{2}$, $\gamma_{4}(G)\cong C_{3}$, $\gamma_{5}(G) = 1$;
 
 	\item  $G'\cong \left< a,b, c, d, e \right>  = \left<c,d \right>\times \left<a,b\right>$, where $\left<c,d\right> \cong C_{p}\times C_{p}$ and \\ $\left<a,b,e| = a^{p} = b^{p}= e^{p}= 1, [b,a] = e\right> $ is abelian group of order $p^{3}$ and exponent $p$, $\gamma_{3}(G)\cong (C_{p})^{3}$, $\gamma_{4}(G)\cong C_{p}$ and $\gamma_{5}(G) = 1$ for $p\geq5$;
 	\item   $G'\cong  (C_{9})^{2}\times (C_{3})^{2}$,  $\gamma_{3}(G)\subseteq G'^{3}$ and $\gamma_{4}(G) = 1$; 
 	\item  $G'\cong C_{9}\times (C_{3})^{4}$, either  $\gamma_{3}(G)\cong C_{3}$,  $\gamma_{4}(G) = 1$, $|G'^{3}\cap \gamma_{3}(G)| = 1$ or $G'^{3} \subseteq \gamma_{3}(G)\cong (C_{3})^{2}$, $\gamma_{4}(G) = 1$; 
 	\item $G'$ is   one of the groups  $S(729, 422)$ or $S(729,502)$, $G'^{3}\subseteq \gamma_{3}(G)\cong (C_{3})^{2}$, $|G'^{3}\cap \gamma_{4}(G)| = 1$, $\gamma_{4}(G)\cong C_{3}$ and $\gamma_{5}(G) = 1$;
 	\item $G'$ is   one of the groups  $S(729,423)$ or $S(729,424)$, $G'^{3}\subseteq \gamma_{3}(G) \cong (C_{3})^{2}$, $\gamma_{4}(G)\cong C_{3}$ and $\gamma_{5}(G) = 1$;
 	\item $G'$ is  one of the groups  $S(729,103)$, $S(729,105)$, $S(729,417)$, $S(729,418)$, $S(729,420)$ or $S(729,421)$,  $G'^{3} = \gamma_{3}(G)\cong (C_{3})^{2}$, $\gamma_{4}(G)\cong C_{3}$ and $\gamma_{5}(G) = 1$;
 	\item  $G'$ is   one of the groups   $S(729,416)$, $S(729,419)$, $S(729,499)$ or $S(729,500)$,  $G'^{3}\subseteq \gamma_{3}(G)\cong (C_{3})^{2}$, $\gamma_{4}(G)\cong C_{3}$ and $\gamma_{5}(G) = 1$;
 	
 	\item $G'\cong C_{9}\times (C_{3})^{6}$, $\gamma_{3}(G)\subseteq G'^{3}\cong C_{3}$ and $\gamma_{4}(G) = 1$;
 	\item $G'\cong  (C_{4})^{2}\times (C_{2})^{3}$  and   $\gamma_{3}(G)\cong G'^{2}$;
 	\item $G'\cong C_{4}\times (C_{2})^{5}$, either $|G'^{2}\cap \gamma_{3}(G)| = 1$, $\gamma_{3}(G)\cong C_{2}$ or $G'^{2}\subseteq \gamma_{3}(G)\cong (C_{2})^{2}$;
 	\item $G'\cong C_{9}\times (C_{3})^{5}$, either $\gamma_{3}(G)\cong C_{3}$, $|G'^{3}\cap \gamma_{3}(G)| = 1$ or $G'^{3}\subseteq \gamma_{3}(G)\cong (C_{3})^{2}$;
 	\item  $G'\cong (C_{p})^{7}$, $|G'^{p}\cap \gamma_{3}(G)| = 1$ and  $\gamma_{3}(G)\cong (C_{p})^{2}$  for $p\geq5$;
 	\item $G'$ is  one of the groups   $S(128,2157)$ to $S(128,2162)$  \sloppy or $S(128,2304)$, $G'^{2}\subseteq \gamma_{3}(G)\cong (C_{2})^{2}$, $\gamma_{4}(G)\cong C_{2}$ and $\gamma_{5}(G) = 1$;
 	\item $G'$ is   one of the groups  $S(128,2323)$ to $S(128,2325)$, $|G'^{2}\cap \gamma_{3}(G) | = 2$, $\gamma_{3}(G)\cong (C_{2})^{2}$, $\gamma_{4}(G)\cong C_{2}$ and $\gamma_{5}(G) = 1$;
 	\item $G'$ is  one of the groups  $S(128,2151)$ to $S(128,2156)$, $S(128,2302)$ or $S(128,2303)$, $G'^{2} = \gamma_{3}(G)\cong (C_{2})^{2}$, $\gamma_{4}(G)\cong C_{2}$ and $\gamma_{5}(G)= 1$;
 	\item $G'$ is   one of the groups $S(128,2320)$ to $S(128,2322)$, $G'^{2}\subseteq \gamma_{3}(G)\cong (C_{2})^{2}$, $\gamma_{4}(G)\cong C_{2}$ and $\gamma_{5}(G) = 1$; 
 	\item  $G'$ is   one of the groups  $S(2187,5874)$, $S(2187,5876)$, \sloppy $S(2187,9102)$ to  $S(2187,9105)$,   $G'^{3} = \gamma_{3}(G)\cong (C_{3})^{2}$, $\gamma_{4}(G)\cong C_{3}$ and $\gamma_{5}(G) = 1$;
 	\item  $G'$ is  one of the groups  $S(2187,9100)$,   $S(2187,9101)$, $S(2187,9306)$ or $S(2187,9307)$, $G'^{3}\subseteq \gamma_{3}(G)\cong (C_{3})^{2}$,  $\gamma_{4}(G)\cong C_{3}$ and $\gamma_{5}(G) = 1$;
 	\item   $G'$  is  one of the groups  $S(2187,5867)$,  $S(2187,5870)$, $S(2187,5872)$ or  $S(2187,9096)$ to $S(2187,9099)$, $G'^3 = \gamma_{3}(G)\cong (C_{3})^{2}$, $\gamma_{4}(G)\cong C_{3}$ and $\gamma_{5}(G) = 1$;
 	\item  $G'$ is  one of the groups  $S(2187,9094)$, $S(2187,9095)$, $S(2187,9303)$ or $S(2187,9304)$,  $G'^{3}\subseteq \gamma_{3}(G)\cong ( C_{3})^{2}$, $\gamma_{4}(G)\cong C_{3}$  and $\gamma_{5}(G) = 1$;

 	\item   $G'\cong <a,b,c,d,e,f,g: a^p = b^p = c^p = d^p =e^p= f^p = g^p = 1,  [b,a] =  c>$,\\ $\gamma_{3}(G)\cong (C_{p})^{2}$, $\gamma_{4}(G)\cong C_{p}$ and $\gamma_{5}(G) = 1$ for $p\geq 3$;
 	 \item $G'\cong <a,b,c,d,e,f,g: a^p = b^p = c^p = d^p =e^p= f^p = g^p = 1, [b,a]= e, [d,c] =e>$, $\gamma_{3}(G)\cong (C_{p})^{2}$, $\gamma_{4}(G)\cong C_{p}$ and $\gamma_{5}(G) = 1$ for $ p\geq 3$;

 	\item $G'\cong  (C_{4})^{3}$,   $\gamma_{3}(G)\subseteq G'^{2}$ and $\gamma_{4}(G) = 1$;
 	\item $G'\cong (C_{4})^{2}\times (C_{2})^{2}$,  either $|G'^{2}\cap \gamma_{3}(G)| =1 $, $\gamma_{3}(G)\cong C_{2}$ or $\gamma_{3}(G)\cong (C_{2})^{2}$ or $G'^{2}\subseteq \gamma_{3}(G)\cong (C_{2})^{3}$;
 	\item $G'\cong (C_{2})^{6}$,  $|G'^{2}\cap \gamma_{3}(G)| = 1$, $\gamma_{3}(G)\cong (C_{2})^3$ and $\gamma_{4}(G) = 1$;
 	\item $G'\cong C_{9}\times( C_{3})^{4}$, either $\gamma_{3}(G)\cong (C_{3})^{2}$, $|G'^{2}\cap \gamma_{3}(G)|=  1$ or  $G'^{3}\subseteq \gamma_{3}(G)\cong (C_{3})^{3}$;
 	\item  $G'\cong (C_{p})^{6}$, $|G'^{p}\cap\gamma_{3}(G)| = 1$, $\gamma_{3}(G)\cong( C_{p})^{3}$, $\gamma_{4}(G)\cong C_{p}$ and  $\gamma_{5}(G) = 1$  for  $p\geq5$;
 	\item $G'$ is   one of the groups  $S(64,199)$ to $S(64,201)$,  $|G'^{2}\cap \gamma_{3}(G)| = 2$, $\gamma_{3}(G)\cong (C_{2})^{2}$, $\gamma_{4}(G)\cong C_{2}$ and $\gamma_{5}(G) = 1$;
 	\item $G'$ is one of the groups  $S(64,264)$ or $S(64,265)$, $|G'^{2}\cap \gamma_{3}(G) |= 1$, $\gamma_{3}(G)\cong (C_{2})^{2}$, $\gamma_{4}(G)\cong C_{2}$ and $\gamma_{5}(G) = 1$;
 	\item $G'$  is  one of the groups   $S(64,56)$ to $S(64,59)$,  $\gamma_{3}(G)\subseteq G'^{2}$, $\gamma_{3}(G)\cong (C_{2})^{2}$,  $\gamma_{4}(G)\cong C_{2}$ and $\gamma_{5}(G) = 1$;
 	
 	\item $G'$ is  one of the groups $S(64,193)$ to $S(64,198)$, $|G'^{2}\cap \gamma_{3}(G)| = 2$, $\gamma_{3}(G)\cong (C_{2})^{2}$, $\gamma_{4}(G)\cong C_{2}$ and $\gamma_{5}(G) = 1$;
 	\item  $G'$ is  one of the groups   $S(64,56)$ to $S(64,59)$, $\gamma_{3}(G)\subseteq G'^{2}$, $\gamma_{3}(G)\cong (C_{2})^{3}$, $\gamma_{4}(G)\cong C_{2}$ and $\gamma_{5}(G) = 1$;
 	\item    $G'$ is  one of the groups  $S(64,193)$ to $S(64,198)$  and $G'^{2}\subseteq \gamma_{3}(G)\cong (C_{2})^{3}$; 
 	\item  $G'$ is  one of the groups  $S(729,103)$ to $S(729, 106)$, $S(729, 416)$ to $S(729,420)$, $S(729, 499)$  or $S(729,500)$, $G'^{3}\subseteq \gamma_{3}(G) \cong (C_{3})^{3}$, $\gamma_{4}(G)\cong C_{3}$ and $\gamma_{5}(G) = 1$;
 	\item  $G'$ is   one of the groups  $S(729,103)$, $S(729, 105)$, $S(729, 417)$, $S(729, 418)$, $S(729,420)$ or $S(729, 421)$, $|G'^{3}\cap \gamma_{3}(G)|= 3$, $\gamma_{3}(G)\cong (C_{3})^{2}$, $\gamma_{4}(G)\cong C_{3}$ and $\gamma_{5}(G)= 1$;
 	\item $G'$ is   one of the groups  $S(729,104)$ or $S(729,106)$,  $ \gamma_{3}(G) \subseteq G'^{2}$, $\gamma_{3}(G)\cong (C_{3})^{2}$, $\gamma_{4}(G)\cong C_{3}$ and $\gamma_{5}(G) = 1$;
 	\item  $G'$ is one of the groups   $S(729,416)$, $S(729,419)$, $S(729,499)$ or $S(729,500)$,  $|G'^{3}\cap \gamma_{3}(G)|= 1$, $\gamma_{3}(G)\cong (C_{3})^{2}$, $\gamma_{4}(G)\cong C_{3}$ and $\gamma_{5}(G) = 1$;
 	\item $G'\cong \phi_{2}(1^5)\times (1)$, $\gamma_{3}(G)\cong (C_{p})^{3}$, $|G'^{p}\cap \gamma_{3}(G)| = 1$,  $\gamma_{4}(G)\cong C_{p}$  and   $\gamma_{5}(G)= 1$  for $p\geq5$; 
 	\item $G'\cong  (C_{4})^{2}\times C_{2}$,  either $|G'^{2}\cap \gamma_{3}(G)| = 1$, $\gamma_{3}(G)\cong (C_{2})^{2}$  or $|G'^{2}\cap \gamma_{3}(G)| = 2$, $\gamma_{3}(G)\cong (C_{2})^{3} $ or $G'^{2}\subseteq \gamma_{3}(G)\cong (C_{2})^{4}$;
 	\item $G'\cong C_{4}\times (C_{2})^{3}$, either $|G'^{2}\cap \gamma_{3}(G)| = 1$, $\gamma_{3}(G) \cong (C_{2})^{3}$ or $G'^{2}\subseteq \gamma_{3}(G)\cong (C_{2})^{4}$;
 	\item $G'\cong (C_{2})^{4}$,  $|G'^{2}\cap \gamma_{3}(G)| = 1$ and $\gamma_{3}(G)\cong (C_{2})^{4}$;
 	\item $G'\cong C_{9}\times (C_{3})^{3}$, either $\gamma_{3}(G)\cong (C_{3})^{3}$, $|G'^{3}\cap \gamma_{3}(G)| = 1$ or $\gamma_{3}(G)\cong (C_{3})^{4}$, $G'^{3}\subseteq \gamma_{3}(G)$;
 	\item $G'\cong (C_{p})^{5}$, $\gamma_{3}(G)\cong (C_{p})^{4}$ and  $|G'^{p}\cap \gamma_{3}(G)| = 1$  for $p\geq5$;
 	\item $G'\cong S(32,2)$,  $|G'^{2}\cap \gamma_{3}(G)| =  4$, $\gamma_{3}(G)\cong (C_{2})^{3}$, $\gamma_{4}(G)\cong C_{2}$ and $\gamma_{5}(G) = 1$;
 	\item $G'\cong  S(243,32)$, $|G'^{3}\cap \gamma_{3}(G)| = 1$, $\gamma_{3}(G)\cong (C_{3})^{3}$ and $\gamma_{4}(G)\cong C_{3}$;
 	\item $G'\cong (C_{p})^{10}$, $\gamma_{3}(G)=  1$ and $|G'^{3}\cap \gamma_{4}(G)| = 1$  for $p>0$;
 	\item $G'\cong (C_{p})^{9}$,  $\gamma_{3}(G)\cong C_{p}$, $|G'^{p}\cap \gamma_{3}(G)| = 1$ and $\gamma_{4}(G)= 1$ for  $p\geq3$; 
 	\item $G'\cong C_{4}\times (C_{2})^{7}$,  $\gamma_{3}(G)\subseteq G'^{2}\cong C_{2}$ and $\gamma_{4}(G)= 1$;
 	\item   $G'\cong (C_{2})^{9}$,  $\gamma_{3}(G)\cong C_{2}$, $|G'^{2}\cap \gamma_{3}(G)| = 1$ and $\gamma_{4}(G)= 1$;
 	\item  $G'\cong (C_{p})^{8}$,  $\gamma_{3}(G)\cong (C_{p})^{2} $,  $|G'^{p}\cap \gamma_{3}(G)| = 1$ and  $\gamma_{4}(G)= 1$ for $p\geq3$;
 	\item  $G'\cong (C_{2})^{8}$,  $|G'^{2}\cap \gamma_{3}(G)| = 1$, $\gamma_{3}(G)\cong (C_{2})^{2}$ and $\gamma_{4}(G)= 1$  for $p\geq5$;
 	\item $G'\cong C_{4}\times (C_{2})^{6}$, either $|G'^{2}\cap \gamma_{3}(G)| = 1$, $\gamma_{3}(G)\cong C_{2}$ or $G'^{2}\subseteq \gamma_{3}(G)\cong (C_{2})^{2}$  for  $p\geq5$;
 	\item $G'\cong (C_{4})^{2}\times (C_{2})^{4}$,  $\gamma_{3}(G)\subseteq G'^{2}$ and $\gamma_{4}(G)= 1$;
 	\item $G'\cong (C_{p})^{7}$,  $\gamma_{3}(G)\cong C_{p}\times C_{p}\times C_{p}$ and $|G'^{p}\cap \gamma_{3}(G)| = 1$  for $p\geq3$;
 	\item $G'\cong (C_{4})^{3}\times C_{2}$,  $\gamma_{3}(G)\subseteq G'^{2}$ and $\gamma_{4}(G)= 1$;
 	\item $G'\cong  (C_{4})^{2}\times (C_{2})^{3}$, either $|G'^{2}\cap \gamma_{3}(G)| = 1$, $\gamma_{3}(G)\cong C_{2}$ or $|G'^{2}\cap \gamma_{3}(G)| = 2$, $\gamma_{3}(G)\cong C_{2}\times C_{2}$ or $G'^{2}\subseteq \gamma_{3}(G)\cong (C_{2})^{3}$, $\gamma_{4}(G)= 1$;
 	\item $G'\cong C_{4}\times (C_{2})^{5}$, either $|G'^{2}\cap \gamma_{3}(G)| = 1$, $\gamma_{3}(G)\cong (C_{2})^{2}$ or $G'^{2}\subseteq \gamma_{3}(G)\cong (C_{2})^{3}$;
 	\item $G'\cong (C_{2})^{7}$,  $|G'^{2}\cap \gamma_{3}(G)| = 1$, $\gamma_{3}(G)\cong (C_{2})^{3}$ and $\gamma_{4}(G)= 1$;
 	\item $G'\cong (C_{p})^{6}$, $|G'^{p}\cap \gamma_{3}(G)| = 1$ and $\gamma_{3}(G)\cong (C_{p})^{4}$ for $p\geq3$;
 	\item $G'\cong (C_{4})^{3}$,  $\gamma_{3}(G)\subseteq G'^{2}$ and $\gamma_{4}(G)= 1$;
 	\item $G'\cong (C_{4})^{2}\times (C_{2})^{2}$, either  $|G'^{2}\cap \gamma_{3}(G)| = 1$, $\gamma_{3}(G)\cong (C_{2})^{2}$ or $|G'^{2}\cap \gamma_{3}(G)| = 2$, $\gamma_{3}(G)\cong (C_{2})^{3}$ or $G'^{2}\subseteq \gamma_{3}(G)\cong (C_{2})^{4}$, $\gamma_{4}(G)= 1$;
 	\item $G'\cong C_{4}\times (C_{2})^{4}$, either  $|G'^{2}\cap \gamma_{3}(G)| = 1$, $\gamma_{3}(G)\cong (C_{2})^{3}$ or $G'^{2}\subseteq \gamma_{3}(G) \cong (C_{2})^{4}$, $\gamma_{4}(G)= 1$;
 	\item $G'\cong (C_{2})^{6}$,  $|G'^{2}\cap \gamma_{3}(G)| = 1$, $\gamma_{3}(G)\cong (C_{2})^{4}$ and $\gamma_{4}(G)= 1$.
 \end{enumerate}

 \end{theorem}

  	\begin{proof}
  		Let $t^{L}(KG) = 10p-8$. Then $l = \frac{t^{L}(KG)-2}{p-1} =  10$. Thus  from  \cite{msbs}, $d_{(11)} = 0$, $d_{(10)} = 0$, $d_{(9)} = 0$ and $d_{(8)} \neq 0$ if and only if $p = 7$, $G'\cong C_{7^{2}}\times (C_{7})^{2}$, $\gamma_{3}(G) \subseteq G'^{7}$ or $ G'\cong C_{7^{2}}\times C_{7}$, $\gamma_{3}(G)\cong C_{7}$, $|\gamma_{3}(G)\cap G'^{7}| = 1$ or $G'\cong C_{7^{2}}\times C_{7}$, $\gamma_{4}(G) \subseteq G'^{7}\subseteq \gamma_{3}(G)\cong (C_{7})^{2}$, $\gamma_{5}(G) = 1$. \\
  		
  		Now if $\mathbf {d_{8} = 0} $,  then $d_{(2)}+2d_{(3)}+3d_{(4)}+4d_{(5)}+5d_{(6)}+6d_{(7)} = 10$. If $d_{7}\neq 0$,  then we have  $d_{(7)} = 1$, $d_{(2)} = 4$ or $d_{(7)} = d_{(3)} = 1$,  $d_{(2)} = 2$ or $d_{(7)} = d_{(2)} = d_{(4)} = 1$ or $d_{(7)} = 1$, $d_{(3)} = 2$ or $d_{(7)} = d_{(5)} =  1$.\\
  		
  		If  $ \mathbf {d_{(7)} = 1} $, then all the above cases are discarded by  Lemma \ref{2l1}. \\
  		
  		Now	if $\mathbf {d_{(7)} = 0} $, then $d_{(2)}+2d_{(3)}+3d_{(4)}+4d_{(5)}+5d_{(6)} = 10$.
  		If $d_{(6)}\neq 0$, then we have the following possibilities: $d_{(6)} =  d_{(2)} = d_{(5)} = 1$ or $d_{(6)} = 1$, $d_{(2)} = 5$ or $d_{(6)} = d_{(3)} = 1$, $d_{(2)} = 3$ or $d_{(6)}  = d_{(4)} = 1$, $d_{(2)} = 2$  or $d_{(6)} =  d_{(2)} = 1$, $d_{(3)} = 2$ or $d_{(6)} = d_{(3)} = d_{(4)} = 1$.\\
  		
  		Let $\mathbf {d_{(6)} =  d_{(2)} = d_{(5)} = 1}$.  Then by Lemma \ref{2l1}(2), $\vartheta_{p'}(4)\geq \vartheta_{p'}(2)$, $\forall p>0$, so $d_{(5)} = 0$. \\
  		
  		Let $\mathbf {d_{(6)} = 1, d_{(2)} = 5}$. If $p\neq 5$, then as   $d_{(2+1)} = 0$, $\vartheta_{p'}(5)\geq \vartheta_{p'}(2)$, hence by Lemma \ref{2l1}(2), $d_{(6)} = 0$.  Now if $p =  5$, then $|G'| = 5^{6}$, $|D_{(6),K}(G)| = |D_{(3),K}(G)| = |D_{(4),K}(G)| = |D_{(5), K}(G)| = 5 $. Therefore, $G'$ is abelian and $|G'^{5}|\leq 5$. We have  $G'\cong C_{5^{2}}\times  (C_{5})^{4} $ or $(C_{5})^{6}$. If $G'\cong C_{5^{2}}\times  (C_{5})^{4} $, then  $G'^{5}\subseteq \gamma_{3}(G)$, $\gamma_{4}(G) = 1$. If $G'\cong (C_{5})^{6}$, then $|G'^{5}\cap \gamma_{3}(G)| = 1$, $\gamma_{3}(G)\cong C_{5}$. \\
  		
  		Let $\mathbf {d_{(6)} = d_{(3)} = 1 , d_{(2)} = 3}$. If $p\neq 5$, then by Lemma \ref{2l1}(2), $d_{(4+1)} = 0$, $\vartheta_{p'}(5)\geq \vartheta_{p'}(4)$, so $d_{(6)} = 0$.  Now if $p =  5$,  then  $|G'| = 5^{5}$, $|D_{(4),K}(G)| = |D_{(5),K}(G)| = |D_{(6),K}(G)| = 5 $, $|D_{(3),K}(G)| = 5^{2}$ and $|G'^{5}| = 5$. Thus   $\gamma_{5}(G) = 1$, $|\gamma_{4}(G)| = 5$ and $|\gamma_{3}(G)| = 5^{2}$. Let $G'$ be  an abelian group, then possible $G'$ are $G'\cong (C_{5^{2}})^{2} \times C_{5}$ or $C_{5^{2}}\times (C_{5})^{3}$ or $(C_{5})^{5}$. If $G'\cong  (C_{5^{2}})^{2}  \times C_{5}$, then $\gamma_{3}(G)\subseteq G'^{2}$. If $G'\cong C_{5^{2}}\times (C_{5})^{3}$, then either $|G'^{5}\cap \gamma_{3}(G)| = 1$, $\gamma_{3}(G)\cong C_{5}$ or $G'^{5}\subseteq \gamma_{3}(G)\cong (C_{5})^{2}$. If $G'\cong (C_{5})^{5}$, then $|G'^{5}\cap \gamma_{3}(G)| = 1$, $\gamma_{3}(G)\cong (C_{5})^{2}$. Now if $G'$ be a non-abelian group, then $G'' = \gamma_{4}(G)\cong C_{5}$, $\gamma_{3}(G)\cong (C_{5})^{2}$ and  $\gamma_{3}(G)\subseteq \zeta(G')$. Thus $|\zeta(G')| = 5^{2}$ or $5^{3}$. If $|\zeta(G')| = 5^{2}$, then $\gamma_{3}(G) = \zeta(G')\cong (C_{5})^{2}$ but from the Table 1,  no such group exists. If $|\zeta(G')| = 5^{3}$, then $\zeta(G')\cong C_{5^{2}}\times C_{5}$ or $(C_{5})^{3}$. Thus possible $G'$ are $S(3125,2)$, $S(3125,40)$, $S(3125,41)$, $S(3125,42)$, $S(3125,43)$, $S(3125,44)$, $S(3125,73)$ and  $S(3125,74)$ and for all these groups $G'^{5}\subseteq \zeta(G')$, $G''\subseteq \zeta(G')$ and $G'^{5}\subseteq \gamma_{3}(G)\cong (C_{5})^{2}$.\\
  		
  		Let  $\mathbf {d_{(6)}  = d_{(4)} = 1, d_{(2)} = 2 } $. Then by Lemma \ref{2l1}, this case is not possible.\\

  		Let  $\mathbf {d_{(6)} =  d_{(2)} = 1, d_{(3)} = 2}$. If $p\neq 5$, then by Lemma \ref{2l1}(2), $d_{(3+1)} = 0$, $\vartheta_{p'}(5)\geq \vartheta_{p'}(3)$, so $d_{(6)} = 0$. If $p = 5$, then $|G'| = 5^{4}$, $|D_{(4), K}(G)| = |D_{(5), K}(G)| = |D_{(6), K}(G)| = 5 $, $|D_{(3), K}(G)| = 5^{3}$ and $|G'^{5}| \leq  5$. Thus $\gamma_{5}(G) = 1$, $|\gamma_{4}(G)| = 5$ and $|\gamma_{3}(G)| = 5^{2}$ or $5^{3}$. Let $G'$ be an abelian group. Then, possible $G'$ are, $(C_{5^2})^{2}$ or $C_{5^2}\times (C_{5})^{2}$. If $G'\cong (C_{5^2})^{2}$, then $|D_{(3),K}(G)| \neq 125$. If $G'\cong  C_{5^2}\times (C_{5})^{2}$, then either $|G'^{5}\cap \gamma_{3}(G)| = 1$, $\gamma_{3}(G)\cong (C_{5})^{2}$ or $G'^{5}=  \gamma_{3}(G)\cong C_{5}$ or $G'^{5} \subseteq  \gamma_{3}(G)\cong (C_{5})^{3}$. Now let $G'$ be a non-abelian group. Then $G'' = \gamma_{4}(G)\cong C_{5}$ and $\gamma_{3}(G)\subseteq \zeta(G')$. If $\gamma_{3}(G)\cong (C_{5})^{2}$, then $\gamma_{3}(G) = \zeta(G')\cong (C_{5})^{2}$, but from the Table 2 of \cite{msbs3}, no such group exists. If $\gamma_{3}(G)\cong (C_{5})^{3}$, then $|\zeta(G')| = 125$.  Thus $G'$ is abelian in this case. \\
  		
    	Let  $\mathbf {d_{(6)} = d_{(3)} = d_{(4)} = 1}$. Since $d_{(1+1)} = 0$, then by Lemma \ref{2l1}(2),  $\vartheta_{p'}(5)\geq \vartheta_{p'}(1)$, so $d_{(6)} = 0$, $\forall p>0$. Thus this case is not possible.\\

  		 Now let $\mathbf {d_{(6)} = 0}$.  If $d_{(5)}\neq 0$,  then we have the following possibilities:   $d_{(3)} = 1$,  $d_{(5)} = 2$ or $d_{(2)} = 2$,   $d_{(5)} = 2$ or $d_{(2)} = 6$,  $d_{(5)} = 1$ or $d_{(4)}= 2$, $d_{(5)}= 1$ or $d_{(2)} = 4$, $d_{(3)} = d_{(5)} = 1$ or  $d_{(2)} = d_{(3)} = 2$,  $d_{(5)} = 1$  or  $d_{(2)} =  d_{(3)} = d_{(4)} =  d_{(5)} = 1$ or $d_{(2)} = 3$,  $d_{(4)} = d_{(5)} = 1$.   \\
  		 
  		 Let  $\mathbf {d_{(3)} = 1, d_{(5)} = 2}$. If $p\neq 2$, then by Lemma \ref{2l1}(2), $\vartheta_{p'}(2)\geq \vartheta_{p'}(1)$. So $d_{(3)} = 0$. If $p = 2$, then  by Lemma \ref{2l1}(1), $1= d_{(3)}\leq d_{(2)} = 0$, so this case  is not possible. \\
  		 
  		 Let  $\mathbf {d_{(2)} = 2, d_{(5)} = 2}$. If $p\neq 2 $, then by Lemma(1.1)(2), $d_{(3+1)} = 0$, $\vartheta_{p'}(4)\geq \vartheta_{p'}(3)$ and thus    $d_{(5)} = 0$. If $p = 2$, then by Lemma \ref{2l1}(1), $d_{(2+1)} = 0$, $d_{(5)} = 0$. Thus this case is not possible.\\
  		 
  		 Let $\mathbf {d_{(2)} = 6, d_{(5)} = 1}$. If $p \neq 2$, then  by Lemma \ref{2l1}(2), $d_{(3)} = 0$,  $\vartheta_{p'}(4)\geq \vartheta_{p'}(3)$ and so $d_{(5)} = 0$. If $p = 2$, then by Lemma \ref{2l1}(1), $d_{(2+1)} = 0$ and so $d_{(5)} = 0$. Thus this case is not possible.\\
  		 
  		 Let $ \mathbf{d_{(4)} = 2, d_{(5)} = 1}$. Then by Lemma \ref{2l1}, this case is not possible.\\
  		 
  		  Let  $\mathbf {d_{(2)} = 4, d_{(3)} =  d_{(5)} = 1}$. If $p \neq 2$, then by Lemma \ref{2l1}(2),  $d_{(3+1)} = 0$, $\vartheta_{p'}(4)\geq \vartheta_{p'}(3)$ and so $d_{(5)} = 0$. If $p = 2$, then $|G'| = 2^6$ $|D_{(3),K}(G)| = 2^2$, $|D_{(5),K}(G)| = |D_{(4),K}(G)| = 2$. Since \sloppy $D_{(6),K}(G) =   G'^{8}\gamma_{3}(G)^{4}\gamma_{4}(G)^{2}\gamma_{6}(G) = 1$, thus $\gamma_{4}(G)\subseteq G'^{4}\gamma_{3}(G)^{2}\cong C_{2}$,  $\gamma_{5}(G) = 1$ and $\gamma_{3}(G)\subseteq \zeta(G')$ and so $G'^{2} \neq 1$. First suppose that   $G'$ is an abelian group. Then  possible $G'$ are $C_{8}\times (C_{2})^{3}$ or $(C_{4})^{2} \times (C_{2})^{2}$ or $C_{4}\times (C_{2})^{4}$. If $G'\cong C_{8}\times (C_{2})^{3}$, then  $\gamma_{3}(G)\subseteq G'^{2}$, $\gamma_{3}(G)\cong C_{4}$. If $G'\cong (C_{4})^{2} \times (C_{2})^{2}$, then $\gamma_{3}(G) \subseteq G'^{2} \cong (C_{2})^{2} $. If $G'\cong C_{4}\times (C_{2})^{4}$, then $G'^{2}\subseteq \gamma_{3}(G)\cong C_{4}$. 
  		 Now let $G'$ is a non-abelian group. Thus $G'^{4}\gamma_{3}(G)^{2}= \gamma_{4}(G) = G'' \cong C_{2}$ and $|\zeta(G')|\leq 2^{4}$. Let $|\zeta(G')| = 4$. Now from  the Table 1 of \cite{sbhcms} possible $G'$ are $S(64,199)$ to $S(64,201)$, $S(64,215)$ to $S(64,245)$, $S(64,264)$ and $S(64,265)$. If $G'$ is  any one of the groups $S(64,199)$ to $S(64,201)$ or $S(64,215)$ to $S(64,245)$, then $\gamma_{3}(G)\subseteq G'^{2}$. If $G'$ is  any one of the groups $S(64,264)$ or $S(64,265)$, then $G'^{2}\subseteq \gamma_{3}(G)\cong C_{4}$ or $|G'^{2}\cap \gamma_{3}(G)|= 1$, $\gamma_{3}(G)\cong C_{2}$. Let $|\zeta(G')| = 8$.  But from the Table  1 of \cite{sbhcms} no  group exists with $|G''| = 2$. Now let $|\zeta(G')| = 16$. From Table 1 of  \cite{sbhcms} possible $G'$ are $S(64,193)$ to $S(64,198)$, $S(64,247)$, $S(64,248)$ and $S(64,261)$ to $S(64,263)$. For  all these groups $G''\subseteq G'^{2}\subseteq \zeta(G')$.  If $G'$ is any one of the groups  $S(64,193)$, $S(64,194)$, $S(64,261)$ or $S(64,262)$, then $G'^{4}\gamma_{3}(G)^{2} = 1$. If $G'$ is any one of the groups $S(64,195)$ to $S(64,198)$, then $\zeta(G') \cong C_{4}\times (C_{2})^{2}$,  $G'^{2}\cong (C_{2})^{2}$ and $G'^{4}\gamma_{3}(G)^{2} = 1$. If $G'$ is any one of the groups $S(64,247)$ to  $S(64,248)$, then $G'^{2} = \gamma_{3}(G) \cong C_{4}$. If $G'$ is $S(64,263)$, then $|G'^{2}\cap \gamma_{3}(G)| = 2$, $\gamma_{3}(G)\cong C_{4}$.\\

  		 Let  $\mathbf {d_{(2)} = d_{(3)} = 2, d_{(5)} = 1}$. If $p \neq 2$, then by Lemma \ref{2l1}(2), $d_{(3+1)} = 0$,  $\vartheta_{p'}(4)\geq \vartheta_{p'}(3)$ and so $d_{(5)} = 0$. If $p = 2$, then $|G'| = 2^{5}$, $|D_{(3),K(G)}| = 2^{3}$, $|D_{(4),K}(G)| = |D_{(5),K}(G)| = 2$. Thus $\gamma_{5}(G) = 1$,\sloppy $\gamma_{4}(G) \subseteq G'^{4}\gamma_{3}(G)^{2} \cong C_{2}$.\
  		 Let $G'$ be abelian. Then possible  $G'$ are $C_{8}\times C_{4}$ or $C_{8}\times (C_{2})^{2}$ or $(C_{4})^{2} \times C_{2}$ or $C_{4}\times (C_{2})^{3} $. If $G'\cong C_{8}\times C_{4}$, then either $G'^{2}\subseteq \gamma_{3}(G) \cong C_{4}\times C_{2}$ or $\gamma_{3}(G)\subseteq G'^{2}$, $\gamma_{3}(G)\cong C_{4}$. If $G'\cong C_{8}\times (C_{2})^{2}$, then $G'^{2}\subseteq \gamma_{3}(G) \cong C_{4}\times C_{2}$. If $G'\cong (C_{4})^{2} \times C_{2}$, then either $|G'^{2}\cap \gamma_{3}(G)| = 2$, $\gamma_{3}(G) \cong C_{4}$ or $|G'^{2}\cap \gamma_{3}(G)| = 4$, $\gamma_{3}(G) \cong C_{4}\times C_{2}$. If $G'\cong C_{4}\times (C_{2})^{3}$, then $|G'^{2}\cap \gamma_{3}(G)| = 2$, $\gamma_{3}(G)\cong C_{4}\times C_{2}$.
  		 
  		  Now  let $G'$ be a non - abelian. Then   $G'^{4}\gamma_{3}(G)^{2} =\gamma_{4}(G) = G'' \cong C_{2}$, $\gamma_{3}(G)\subseteq \zeta(G')$ and $|\zeta(G')| \leq 2^3$.  If $|\zeta(G')| = 4$, then from the Table 2 of  \cite{rrm},  no group exists with $|G''| = 2$. If $|\zeta(G')| = 8$, then possible $G'$ are $S(32,2)$,  $S(32,4)$, $S(32,5)$, $S(32,12)$, $S(32,22)$ to $S(32,26)$, $S(32,37)$ to $S(32,38)$ and $S(32,46)$ to $S(32,48)$ (see Table 2 of \cite{rrm}). If $G'$ is $S(32,2)$, then $G'^{4}\gamma_{3}(G)^{2} = 1$, which is not possible. If $G'$ is  any one of the groups $S(32,4)$, $S(32,5)$  or $S(32,12)$, then $\zeta(G') \cong C_{4}\times C_{2}$ and $\gamma_{3}(G)\subseteq G'^{2} \cong C_{4}\times C_{2}$. If $G'$ is any one of the groups  $S(32,22)$ to $S(32,26)$, then $|G'^{2}\cap \gamma_{3}(G) |= 2$, $\gamma_{3}(G) \cong C_{4}$ or $G'^{2}\subseteq \gamma_{3}(G)\cong C_{4}\times C_{2}$. If $G'$ is any one of the groups  $S(32,37)$ or $S(32,38)$, then either $|G'^{2}\cap \gamma_{3}(G)| = 2$, $\gamma_{3}(G)\cong (C_{2})^2$ or $G'^{2}\subseteq \gamma_{3}(G) \cong C_{4}\times C_{2}$. If $G'$ is any one of the groups $S(32,46)$ to $S(32,48)$, then either $|G'^{2}\cap \gamma_{3}(G)| = 1$, $\gamma_{3}(G)\cong C_{4}$ or $G'^{2}\subseteq \gamma_{3}(G) \cong C_{4}\times C_{2}$.   \\
  		
  		Let  $\mathbf{ d_{(2)} =  d_{(3)} = d_{(4)} =  d_{(5)} = 1}$.  Then $|G'| = p^{4}$,  $|D_{(3),K}(G)| = p^{3}$, $|D_{(4),K}(G)| = p^{2}$, $|D_{(5),K}(G)| = p$ and $D_{(6),K}(G) = 1$, $\forall p> 0$. Let $G'$ is an abelian group. If $p\geq 5$, then $G'\cong (C_{p})^{4}$,   $|G'^{p}\cap \gamma_{3}(G)| = 1$, $\gamma_{3}(G)\cong (C_{p})^{3}$, $\gamma_{4}(G)\cong (C_{p})^{2}$ and $\gamma_{5}(G) \cong C_{p}$.  If $p = 3$, then $D_{(6),K}(G) = 1$ leads to $G'^{9} = 1$, so either  $G' \cong  C_{9}\times (C_{3})^{2}$ or  $(C_{3})^{4}$. If $G'\cong C_{9}\times (C_{3})^{2}$, then  either $|G'^{2}\cap \gamma_{3}(G)| = 1$, $\gamma_{3}(G)\cong (C_{3})^{2}$ or $G'^{3}\subseteq \gamma_{3}(G) \cong (C_{3})^{3}$. If $G'\cong (C_{3})^{4}$, then $|G'^{3}\cap \gamma_{3}(G)| = 1$, $\gamma_{3}(G)\cong (C_{3})^{3}$. If $p = 2$, then $D_{(6),K}(G) = 1$ leads to $G'^{8} = 1$, so $G'\cong C_{8}\times C_{2}$ or $(C_{4})^{2}$ or $C_{4}\times (C_{2})^{2}$. If $G'\cong  C_{8}\times C_{2}$, then  either $\gamma_{3}(G)\cong C_{2}$, $|G'^{2}\cap \gamma_{3}(G)| = 1$ or $|G'^{2}\cap \gamma_{3}(G)| = 2$, $\gamma_{3}(G)\cong (C_{2})^{2}$  or $G'^{2}\subseteq \gamma_{3}(G) \cong C_{4}\times C_{2}$.  Clearly $G'\cong (C_{4})^{2}$ is not possible as $|D_{(3),K}(G)| < 2^{3}$. If $G'\cong C_{4}\times (C_{2})^{2}$, then $G'^{2}\subseteq \gamma_{3}(G)\cong C_{4}\times C_{2}$.
  		Let $G'$ be a  non-abelian and $p\geq 5$. Then  $\gamma_{6}(G) = 1$, $\gamma_{5}(G)\cong C_{p}$, $\gamma_{4}(G)\cong (C_{p})^{2}$,  $\gamma_{3}(G)\cong (C_{p})^{3}$ and  $\zeta(G')\cong (C_{p})^{2}$. Thus possible  $G'$ is $((C_{p}\times C_{p})\rtimes C_{p})\times C_{p}$, (see \cite{gjg}).  If $p = 3$, then  $\gamma_{5}(G) = 1$, $|\gamma_{4}(G)| = 3$, $|\gamma_{3}(G)| = 3^{2} $ or $3^{3}$. Now $ \gamma_{3}(G) \subseteq \zeta(G')$ and $|\zeta(G')| \leq 3^{2}$, so $|\gamma_{3}(G)| \neq 3^{3}$. Thus $\gamma_{3}(G) = \zeta(G') \cong C_{3}\times C_{3}$ and $G'' = \gamma_{4}(G) \cong C_{3} $. Therefore, from Table 2 of \cite{msbs3} the possibilities for $G'$ are $S(81,3)$, $S(81,4)$, $S(81,12)$ and $S(81,13)$, but then $|D_{(3),K}(G)| < 3^{3}$.  If $p = 2$, then   $\gamma_{5}(G) = 1$, $ G'' = \gamma_{4}(G) = G'^{4}\gamma_{3}(G)^{2} \cong C_{2}$, $\gamma_{3}(G) \subseteq \zeta(G') $ and  $|\zeta(G')| = 2^{2}$. So  possible $G'$ are $S(16,3)$, $S(16,4)$ and $S(16,11)$ to  $S(16,13)$. But for these groups $|D_{(3),K}(G)| \neq 2^{3}$ (see Table 1 of  \cite{msbs3}).  \\
  		 
  		 Let $\mathbf {d_{(2)} = 3, d_{(4)} = d_{(5)} = 1}$. If $p\neq 2$, then by Lemma \ref{2l1}(2), $d_{(2+1)} = 0$,  $\vartheta_{p'}(4)\geq \vartheta_{p'}(2)$ and so $d_{(5)} = 0$. If $p = 2$, then  by Lemma \ref{2l1}(1), as $d_{(3)} = 0$, so $d_{(5)} = 0$. Thus this case is not possible.\\

  		  Now let $\mathbf  {d_{(5)} = 0}$.  If $d_{(4)}\neq 0$, then we have the following possibilities:  $d_{(4)} = 3$, $d_{(2)} = 1$ or $d_{(4)} =  d_{(2)} = 2$, $d_{(3)} = 1$ or $d_{(4)} = 2$, $d_{(2)} = 4$ or $d_{(4)} = 1$, $d_{(2)} = 7$ or $d_{(4)} = d_{(3)} = 1$, $d_{(2)} = 5$ or $d_{(4)} = 1$, $d_{(2)} = 3$, $d_{(3)} = 2$ or $d_{(4)} = 1, d_{(2)} = 1$, $d_{(3)} = 3$ or $d_{(4)} = d_{(3)} = 2$.\\
  		  
  		  Let  $\mathbf {d_{(4)} = 3, d_{(2)} = 1}$. Now   $p \neq 3$ is not possible by Lemma \ref{2l1}. Thus  $p = 3$. Now $|G'| = 3^{4}$, $|D_{(4),K}(G)| =  |D_{(3),K}(G)| = 3^{3}$, $D_{(5),K}(G) = 1$. Let $G'$ be an abelian group, then  $G'\cong (C_{9})^{2}$ or $C_{9}\times (C_{3})^{2}$ or $(C_{3})^{4}$. If $G'\cong (C_{9})^{2}$, then $G'^{3} = \gamma_{3}(G) \cong (C_{3})^{3}$. If $G'\cong C_{9}\times (C_{3})^{2}$, then either $|G'^{3}\cap \gamma_{3}(G)| = 1$, $\gamma_{3}(G) \cong (C_{3})^{2}$ or $G'^{3}\subseteq \gamma_{3}(G) \cong (C_{3})^{3} $. If $G'\cong  (C_{3})^{4}$, then $|G'^{3}\cap \gamma_{3}(G)| = 1$, $\gamma_{3}(G) \cong (C_{3})^{3} $. Now let  $G'$ be a non-abelian group. Then $G'' = \gamma_{4}(G) \cong C_{3}$, $\gamma_{3}(G) \subseteq \zeta(G')$. So $\gamma_{3}(G) = \zeta(G') \cong (C_{3})^{2}$. Hence possible $G'$  are $S(81,3)$,  $S(81,4)$, $S(81,12)$ and $S(81,13)$. But for these groups $|D_{(3),K}(G)| < 3^{3}$ (see Table 2 of \cite{msbs3}).   \\
  		  
  		  Let $\mathbf  {d_{(4)} = d_{(2)} = 2, d_{(3)} = 1}$. Then  $|G'| = p^{5}$, $|D_{(3),K}(G)| = p^{3}$ and \sloppy  $|D_{(4),K}(G)| = p^{2}$, $\forall p >0$.  Let $G'$ be an abelian group and $p\geq 5$. Now   $D_{(5),K}(G) = 1$ leads to  $G'^{p} = 1$, so $G'\cong (C_{p})^{5}$, $\gamma_{3}(G) \cong (C_{p})^{3}$ and $|G'^{p} \cap \gamma_{3}(G)| = 1$. If $p = 3$, then $G'\cong (C_{9})^{2}\times C_{3}$ or $C_{9}\times (C_{3})^{3}$ or $(C_{3})^{5}$. If $G'\cong (C_{9})^{2} \times C_{3}$, then either $G'^{3}\subseteq \gamma_{3}(G)\cong (C_{3})^{3}$ or $|G'^{3}\cap \gamma_{3}(G)| = 3$, $\gamma_{3}(G)\cong (C_{3})^{2}$ or $|G'^{3}\cap \gamma_{3}(G)| = 1$, $\gamma_{3}(G)\cong C_{3}$.  If $G'\cong C_{9}\times (C_{3})^{3}$, then  either $|G'^{3}\cap \gamma_{3}(G)| = 1$, $\gamma_{3}(G)\cong (C_{3})^{2}$ or $G'^{3}\subseteq  \gamma_{3}(G) \cong (C_{3})^{3}$. If $G'\cong (C_{3})^{5}$, then $|G'^{3}\cap \gamma_{3}(G)| = 1$, $\gamma_{3}(G) \cong (C_{3})^{3}$. If $p = 2$, then $G'\cong (C_{4})^{2} \times C_{2}$ or $C_{4}\times (C_{2})^{3}$ or $(C_{2})^{5}$. If $G'\cong  (C_{4})^{2}\times C_{2}$, then  either $G'^{2}\subseteq \gamma_{3}(G)\cong (C_{2})^{3}$ or $|G'^{2}\cap \gamma_{3}(G)| = 2$, $\gamma_{3}(G) \cong (C_{2})^{2} $ or $|G'^{2}\cap \gamma_{3}(G)| = 1$, $\gamma_{3}(G)\cong C_{2}$. If $G'\cong C_{4}\times (C_{2})^{3}$, then either $G'^{2}\subseteq \gamma_{3}(G)\cong (C_{2})^{3}$ or $|G'^{2}\cap \gamma_{3}(G)| = 1$, $\gamma_{3}(G)\cong (C_{2})^{2} $. If $G'\cong (C_{2})^{5}$, then $|G'^{2}\cap \gamma_{3}(G)| = 1$, $\gamma_{3}(G)\cong (C_{2})^{3}$.  Let $G'$ be a non - abelian group. If $p = 2$, then $D_{(5),K}(G) =G'^{4}\gamma_{3}(G)^{2}\gamma_{5}(G) = 1$ leads to $|\gamma_{4}(G)| = 2$ or  $4$ and $|\gamma_{3}(G)| = 4$ or $8$. First, let $|\gamma_{4}(G)| =2$. Then $ G'' = \gamma_{4}(G) \cong  C_{2}$, so $\gamma_{3}(G) \cong (C_{2})^{2}$ or $(C_{2})^{3}$ and $|\zeta(G')| = 2^{2}$ or $2^{3}$. If $|\zeta(G')| = 4$, then from the Table 2 of \cite{rrm},  $|G''| \neq 2$. If $|\zeta(G')| = 8$, then $\zeta(G')\cong C_{4}\times C_{2}$ or $(C_{2})^{3}$. Therefore  possible $G'$ are $S(32,2)$, $S(32,4)$, $S(32,5)$, $S(32,12)$, $S(32,22)$ to $S(32,26)$, $S(32,37)$  and $S(32,46)$ to $S(32,48)$ (see table 2 of \cite{rrm}).  If $G'$ is any one of the groups   $S(32,4)$, $S(32,5)$, $S(32,12)$ or $S(32,37)$, then $G'^{4} \neq 1$. If $G'$ is $S(32,2)$, then $\gamma_{3}(G)\subseteq G'^{2}$. If $G'$ is  any one of the groups $S(32,22)$ to $S(32,26)$, then either $|G'^{2}\cap \gamma_{3}(G)| = 2$, $\gamma_{3}(G)\cong (C_{2})^{2}$ or $G'^{2}\subseteq \gamma_{3}(G)\cong (C_{2})^{3}$. If $G'$ is any one of the groups  $S(32,46)$ to $S(32,48)$, then  either $|G'^{2}\cap \gamma_{3}(G)| = 1$, $\gamma_{3}(G)\cong (C_{2})^{2}$ or $G'^{2}\subseteq \gamma_{3}(G)\cong (C_{2})^{3}$. Now for $|\gamma_{4}(G)| = 4$,  $\gamma_{4}(G)\cong (C_{2})^{2}$, $\gamma_{3}(G) = \zeta(G') \cong (C_{2})^{3}$. But from the Table 2 of \cite{rrm}, no such group exists with $|G''| = 4$.   Let $p= 3$, then $D_{(5),K}(G) = 1$ leads to $G'^{9} = 1$ and  $|D_{(4),K}(G)| = 3^{2}$ leads  to $|\gamma_{4}(G)| = 3$ or $3^{2}$. So  $G'' = \gamma_{4}(G) \cong C_{3}$ and $\gamma_{3}(G)\cong (C_{3})^{2}$ or $(C_{3})^{3}$. Thus $|\zeta(G')| = 9$ or 27. First let $|\zeta(G')| = 9$, then from the Table 5 of \cite{msbs3},  no such  group exists. If $|\zeta(G')| = 27$, then possible $G'$ are $S(243,2)$,   $S(243,32)$ to  $S(243,36)$ and $S(243,62)$ to  $S(243,64)$ (see  Table 5 of \cite{msbs3}). Now $\gamma_{4}(G)\subseteq G'^{3}\gamma_{3}(G)^{3} \cong (C_{3})^{2}$ and $\gamma_{3}(G)^{3} = 1$.  Hence $|G'^{3}| = 9$.   If $G'$ is  one of the group from  $S(243,32)$,  $S(243,35)$  or $S(243,62)$ to $S(243,64)$, then $|G'^{3}| \neq 9$.   If $G'$ is any one of the groups $S(243,2)$, $S(243,33)$, $S(243,34)$ or $S(243,36)$, then either $G'^{3}\subseteq \gamma_{3}(G)\cong (C_{3})^{3}$ or  $|G'^{3}\cap \gamma_{3}(G)| = 3$, $\gamma_{3}(G)\cong (C_{3})^{2}$. For $|\gamma_{4}(G)| = 9$,  $G'' = \gamma_{4}(G)\cong (C_{3})^{2}$ and $\gamma_{3}(G) = \zeta(G') \cong (C_{3})^{3}$. But  from the Table 3 no such group exists.  For $p\geq 5$,  $\gamma_{3}(G)\cong (C_{p})^{3}$ or  $(C_{p})^{2}$,  $\gamma_{4}(G)\cong (C_{p})^{2}$ or $C_{p}$ and $\gamma_{5}(G)=  1$. If $G'' = \gamma_{4}(G)\cong (C_{p})^{2}$ and   $\gamma_{3}(G)\cong (C_{p})^{3}$, then $\gamma_{3}(G) \subseteq \zeta(G')$ and hence $\gamma_{3}(G)=  \zeta(G')\cong (C_{p})^{3}$.  But from  \cite{zsx} no such group exists. If $G''= \gamma_{4}(G)\cong C_{p}$,   $\gamma_{3}(G)\cong (C_{p})^{3}$ or  $(C_{p})^{2}$ and $\gamma_{5}(G) = 1$, then   $|\zeta(G')| = p^{2}$ or $p^{3}$. Let $\zeta(G')\cong (C_{p})^{2}$, then from  \cite{zsx} no such group exists.  Now let $\gamma_{3}(G) = \zeta(G')\cong (C_{p})^{3}$, then from \cite{zsx}, $G'\cong \left< a,b, c, d, e \right>  = \left<c,d \right>\times \left<a,b\right>$, where $\left<c,d\right> \cong C_{p}\times C_{p}$ and $\left<a,b,e| = a^{p} = b^{p}= e^{p}= 1, [b,a] = e\right>  $ is a non - abelian group of order $p^{3}$ and exponent $p$.\\
  		 
  		  Let $\mathbf  {d_{(4)} = 2, d_{(2)} = 4}$.  If $p\neq 3$ and  $d_{(3)} = 0$, then by Lemma \ref{2l1}(2),  $\vartheta_{p'}(3)\geq \vartheta_{p'}(2)$,  so $d_{(4)} = 0$. If $p = 3$, then $|G'| = 3^{6}$, $|D_{(3),K}(G)| = |D_{(4),K}(G)| = 3^{2}$, $D_{(5),K}(G) = 1$ and $G'^{3} \neq 1$. Let $G'$ be an abelian group. So possible $G'$ are $ (C_{9})^{2}\times (C_{3})^{2}$ or $C_{9}\times (C_{3})^{4}$. If $G'\cong  (C_{9})^{2}\times (C_{3})^{2}$, then $\gamma_{3}(G)\subseteq G'^{3}$. If $G'\cong C_{9}\times (C_{3})^{4}$, then either $\gamma_{3}(G)\cong C_{3}$, $|G'^{3}\cap \gamma_{3}(G)| = 1$ or $G'^{3} \subseteq \gamma_{3}(G)\cong (C_{3})^{2}$.   Let $G'$ be a  non - abelian group. Now  $G'' = \gamma_{4}(G) \cong C_{3}$, $\gamma_{3}(G)\cong (C_{3})^{2}$ and $G'^{3}\subseteq \gamma_{3}(G)$ . Hence either $G'^{3} = \gamma_{3}(G) \cong (C_{3})^{2}$ or $G'^{3}\cong C_{3}$ and $|G'^{3}\cap \gamma_{4}(G)| = 1$. As $\gamma_{3}(G)\subseteq \zeta(G')$, so $|\zeta(G')| = 3^2$ or $3^3$ or $3^4$. If $|\zeta(G')| = 3^{2}$, then $\gamma_{3}(G) = \zeta(G')\cong (C_{3})^{2}$. Hence from the Table 6 of \cite{msbs3}, possible $G'$ are  $S(729,422)$  to $S(729,424)$ and  $S(729,502)$. If $G'$ is any  one of the groups $S(729,422)$ or $S(729,502)$, then $G'^{3}\cong C_{3}$, $|G'^{3}\cap \gamma_{4}(G)| = 1$. If $G'$ is any one of the groups  $S(729,423)$ or $S(729,424)$, then $G'^{3}\subseteq \gamma_{3}(G)\cong (C_{3})^{2}$, $\gamma_{4}(G) \cong C_{3}$. If $|\zeta(G')| = 3^3$, then from the Table 6 of \cite{msbs3}, no such group exists. Now if $|\zeta(G')| = 3^4$, then possible $G'$ are $S(729,103)$, $S(729,105)$, $S(729,416)$ to $S(729,421)$ and $S(729,499)$ to  $S(729,500)$. If $G'$ is any one of the groups $S(729,103)$, $S(729,105)$, $S(729,417)$, $S(729,418)$, $S(729,420)$ or $S(729,421)$, then $G'^{3} = \gamma_{3}(G)\cong (C_{3})^{2}$. If $G'$ is any one of the groups  $S(729,416)$, $S(729,419)$, $S(729,499)$ or $S(729,500)$, then $G'^{3}\subseteq \gamma_{3}(G)\cong (C_{3})^{2}$ (see Table 6 of \cite{msbs3}).\\
  		  
  		  Let $\mathbf  {d_{(4)} = 1, d_{(2)} = 7}$. If $p\neq 3$ and $d_{(2+1)} = 0$, then by Lemma \ref{2l1}(2), $\vartheta_{p'}(3)\geq \vartheta_{p'}(2)$,  so $d_{(4)} = 0$. If $p = 3$, then $|G'| = 3^{8}$, $|D_{(3),K}(G)| = |D_{(4),K}(G)| = 3$ and $\gamma_{4}(G)= 1$. Thus $G'$ is abelian in this case. Now $|D_{(4),K}(G)| =3 $ leads to $|G'^{3}| = 3$. So only possible $G'$ is $C_{9}\times (C_{3})^{6}$, $\gamma_{3}(G)\subseteq G'^{3}\cong C_{3}$.\\
  		  
  		  Let $ \mathbf  {d_{(4)} = d_{(3)} = 1, d_{(2)} = 5} $. Now $|G'| = p^{7}$, $|D_{(3),K}(G)| = p^{2}$, $|D_{(4),K}(G)| = p$, $|D_{(5),K}(G)| = 1$, for all $p>0$.  Let $G'$ be an abelian group. If $p = 2$, then $|G'^{2}| = 2$ or $4$. So possible $G'$ are $(C_{4})^{2}\times (C_{2})^{3}$ or $C_{4}\times (C_{2})^{5}$. If $G'\cong  (C_{4})^{2}\times (C_{2})^{3}$,  then  $\gamma_{3}(G)\subseteq G'^{2}$. If $G'\cong C_{4}\times (C_{2})^{5}$, then either $|G'^{2}\cap \gamma_{3}(G)| = 1$, $\gamma_{3}(G)\cong C_{2}$ or $G'^{2}\subseteq \gamma_{3}(G)\cong (C_{2})^{2}$. If $p = 3$, then $|D_{(4),K}(G)| = 3$ leads to $|G'^{3}| = 3$. So  $G'\cong C_{9}\times (C_{3})^{5}$, either $\gamma_{3}(G)\cong C_{3}$, $|G'^{3}\cap \gamma_{3}(G)| = 1$ or $G'^{3}\subseteq \gamma_{3}(G)\cong (C_{3})^{2}$. If $p \geq 5$, then $|D_{(4),K}(G)| = p$ leads to $G'^{p} = 1$ and $G'\cong (C_{p})^{7}$, $|G'^{p}\cap \gamma_{3}(G)| = 1$, $\gamma_{3}(G)\cong (C_{p})^{2}$.  Let $G'$ be a non - abelian group. Then for $p= 2$,  $G'' = \gamma_{4}(G)\cong C_{2}$, $\gamma_{3}(G)\cong (C_{2})^{2}$ and  $G'^{2}\subseteq \gamma_{3}(G)$. Since $\gamma_{3}(G)\subseteq \zeta(G')$, therefore  $|\zeta(G')| \geq 4$. If   $|\zeta(G')| = 4$ or 16, then from the Table  4 of \cite{rrm} no such group exists. If $|\zeta(G')| = 8$, then  possible $G'$ are $S(128,2157)$ to $S(128,2162)$, $S(128,2304)$ and  $S(128,2323)$ to $S(128,2325)$ (see table 4  of \cite{rrm}).  If $G'$ is any one of the groups $S(128,2157)$ to $S(128,2162)$ or $S(128,2304)$, then $G'^{2} =  \gamma_{3}(G)\cong (C_{2})^{2}$. If $G'$ is any one of the groups $S(128,2323)$ to $S(128,2325)$, then $G'^{2} \subseteq  \gamma_{3}(G)\cong (C_{2})^{2}$. Let $|\zeta(G')| = 32$, then from the Table 4 of \cite{rrm}, possible $G'$ are $S(128,2151)$ to $S(128,2156)$, $S(128,2302)$, $S(128,2303)$ and  $S(128,2320)$ to $S(128,2322)$.  If $G'$ is any one of the groups $S(128,2151)$ to $S(128,2156)$, $S(128,2302)$ or $S(128,2303)$, then $G'^{2} = \gamma_{3}(G)\cong (C_{2})^{2}$. If $G'$ is any one of the groups $S(128,2320)$ to $S(128,2322)$, then $G'^{2}\subseteq \gamma_{3}(G)\cong (C_{2})^{2}$.  If $p = 3$, then $D_{(5),K}(G) = G'^{9}\gamma_{3}(G)^{3}\gamma_{5}(G) = 1$ leads to $\gamma_{5}(G) = 1$.  Now $G'' = \gamma_{4}(G)\cong C_{3}$, $\gamma_{3}(G)\cong (C_{3})^{2}$ and $|\zeta(G')|\geq 3^2$.     First let $exp(G') = 9$. If  $|\zeta(G')| = 3^2$ and $3^4$, then from the Table 2, no such group exists.  If $|\zeta(G')| = 3^3$, then  possible  $G'$ are $S(2187,5874)$, $S(2187,5876)$, $S(2187,9100)$ to  $S(2187,9105)$  and $S(2187,9306)$ to $S(2187,9307)$ (see Table 2). If $G'$ is any one of the groups  $S(2187,5874)$, $S(2187,5876)$,  $S(2187,9102)$ to  $S(2187,9103)$,  $S(2187,9104)$ or  $S(2187,9105)$, then $G'^{3} = \gamma_{3}(G)\cong (C_{3})^{2}$.  If $G'$  is any  one of the groups  $S(2187,9100)$ to $S(2187,9101)$,$S(2187,9306)$ or $S(2187,9307)$, then $G'^{3}\subseteq \gamma_{3}(G)\cong (C_{3})^{2}$. Now let $|\zeta(G')| = 3^5$. So possible $G'$  are $S(2187,5867)$, $S(2187,5870)$, $S(2187,5872)$, $S(2187,9094)$ to $S(2187,9099)$ and  $S(2187,9303)$ to $S(2187,9304)$ (see Table 2 ). If $G'$ is any one of the groups $S(2187,5867)$,  $S(2187,5870)$, $S(2187,5872)$ or  $S(2187,9096)$ to $S(2187,9099)$, then $G'^3 = \gamma_{3}(G)\cong (C_{3})^{2}$. If $G'$ is  any one of the groups $ S(2187,9094)$ to $S(2187,9095)$, $S(2187,9303)$ or $S(2187,9304)$, then $G'^{3}\subseteq \gamma_{3}(G)\cong (C_{3})^{2}$.  For $p\geq 5$, $D_{(5),K}(G) = G'^{p}\gamma_{3}(G)^{p}\gamma_{5}(G) = 1$ leads to  $exp(G') = p$. Now let $exp(G') = p$, for $p\geq 3$. Therefore $G'' =\gamma_{4}(G)\cong C_{p}$ and $\gamma_{3}(G)\cong (C_{p})^{2}$. Therefore possible $G'$ are $<a,b,c,d,e,f,g: a^p = b^p = c^p = d^p =e^p= f^p = g^p = 1,  [b,a] =  c>$ and $<a,b,c,d,e,f,g: a^p = b^p = c^p = d^p =e^p= f^p = g^p = 1, [b,a]= e, [d,c] =e>$ and for these groups,  we have $\gamma_{3}(G)\cong (C_{p})^{2}$ (see\cite{wid}). \\ 
  		 
  		  Let $\mathbf  {d_{(4)} = 1, d_{(2)} = 3, d_{(3)} = 2}$. Now $|G'|= p^{6}$, $|D_{(4),K}(G)| = p$, $|D_{(3),K}(G)| = p^{3}$, for all $p>0$. Let $G'$ be an abelian group. For $p=2$,   $G'\cong (C_{4})^{3}$ or $(C_{4})^{2} \times (C_{2})^{2}$ or $C_{4}\times (C_{2})^{4}$ or $(C_{2})^{6}$. If $G'\cong  (C_{4})^{3}$, then $\gamma_{3}(G)\subseteq G'^{2}$. If $G'\cong (C_{4})^{2}\times (C_{2})^{2}$, then  either $|G'^{2}\cap \gamma_{3}(G)| =1 $, $\gamma_{3}(G)\cong C_{2}$ or $|G'^{2} \cap \gamma_{3}(G)| = 2$, $\gamma_{3}(G)\cong (C_{2})^{2}$ or $G'^{2}\subseteq \gamma_{3}(G)\cong (C_{2})^{3}$. If $G'\cong (C_{2})^{6}$, then $|G'^{2}\cap \gamma_{3}(G)| = 1$, $\gamma_{3}(G)\cong (C_{2})^{3}$. For  $p = 3$, $|D_{(4),K}(G)| = 3$ leads to $|G'^{3}| = 3$. Hence  $G'\cong C_{9}\times( C_{3})^{4}$.   For this group, either $\gamma_{3}(G)\cong (C_{3})^{2}$, $|G'^{2}\cap \gamma_{3}(G)|=  1$ or  $G'^{3}\subseteq \gamma_{3}(G)\cong (C_{3})^{3}$. For $p\geq 5$,  $D_{(5),K}(G) = 1$ leads to $G'^{p} = 1$  and  $G'\cong (C_{p})^{6}$, $|G'^{p}\cap\gamma_{3}(G)| = 1$, $\gamma_{3}(G)\cong( C_{p})^{3}$.  Let $G'$ be a non-abelian group. For $p =2$,  $G'' = \gamma_{4}(G)\cong C_{2}$ and $\gamma_{3}(G)\cong (C_{2})^{2}$ or $(C_{2})^{3}$. If $\gamma_{3}(G)\cong (C_{2})^{2}$, then $|\zeta(G')| = 4$, 8 or 16. Let $|\zeta(G')| = 4$. Then   $\gamma_{3}(G) = \zeta(G')\cong (C_{2})^{2}$. Therefore, possible $G'$ are $S(64,199)$ to $S(64,201)$ and  $S(64,264)$ to $S(64,265)$ (see Table 1 of \cite{sbhcms}).  If  $G'$ is any one of the groups  $S(64,199)$ to $S(64,201)$, then $|G'^{2}\cap \gamma_{3}(G)| = 2$. If $G'$ is any one of the groups $S(64,264)$ or $S(64,265)$, then $|G'^{2}\cap \gamma_{3}(G) |= 1$.  Let $|\zeta(G')| = 8$, therefore from the Table 1 of \cite{sbhcms}, no such group exists.  Let $|\zeta(G')| = 16$. Then  for $|G'^{2}| = 8$, possible $G'$ are $S(64,56)$ to $S(64,59)$. For these groups, $\gamma_{3}(G)\subseteq G'^{2}$, $\gamma_{3}(G)\cong (C_{2})^{2}$. For  $|G'^{2}| = 4$, possible $G'$ are $S(64,193)$ to $S(64,198)$ and for these groups $|G'^{2}\cap \gamma_{3}(G)| = 2$. For $|G'^{2}| = 2$,  possible $G'$ are $S(64,261)$ to $S(64,263)$.  For these groups  $2= |G''\cap G'^{2}| \leq |G'^{2}\cap \gamma_{3}(G)| = 1$. If $|\gamma_{3}(G)| = 8$,  then  $|\zeta(G')| = 8$ or 16. Let $|\zeta(G')| = 8$, then  from the  Table  1 of \cite{sbhcms}  no such group exists. Let $|\zeta(G')| = 16$. Now  for $|G'^{2}| = 8$,  possible $G'$ are $S(64,56)$ to $S(64,59)$ and for these groups $\gamma_{3}(G)\subseteq G'^{2}$. For $|G'^{2}| = 4$,   possible $G'$ are $S(64,193)$ to $S(64,198)$, and for these groups  $G'^{2}\subseteq \gamma_{3}(G)\cong (C_{2})^{3}$. For $|G'^{2}| = 2$, no  group exists ( see Table 1  of \cite{sbhcms}). For $p= 3$, $G'' = \gamma_{4}(G) \cong C_{3}$. So $\gamma_{3}(G)\cong (C_{3})^{2}$ or $(C_{3})^{3}$ and  $|\zeta(G')|\geq 3^{2}$. Let $\gamma_{3}(G)\cong (C_{3})^{2}$ and  $|\zeta(G')| = 9$.  Then possible $G'$ are $S(729,422)$ to  $S(729,424)$ and $S(729,502)$ (see Table 6 of \cite{msbs3}). But then $|D_{(3),K}(G)| \neq 3^3$. If $|\zeta(G')| = 3^3$, then no  group exists (see Table 6 of \cite{msbs3}). Let  $|\zeta(G')|= 3^4$. If $|\gamma_{3}(G)| = 3^3$, then $G'$ is any one of the groups  $S(729,103)$ to $S(729,106)$, $S(729,416)$ to $S(729,420)$ or $S(729,499)$ to $S(729,500)$. For all these groups $G'^{3}\subseteq \gamma_{3}(G)\cong (C_{3})^{3}$. Let $|\gamma_{3}(G)| = 3^{2}$, then possible $G'$ are $S(729,103)$ to $S(729,106)$, $S(729,416)$ to $S(729,421)$ and $S(729,499)$ to $S(729,500)$. If $G'$ is any one of the groups $S(729,103)$,  $S(729,105)$, $S(729,417)$, $S(729,418)$, $S(729,420)$ or $S(729,421)$, then $|G'^{3}\cap \gamma_{3}(G)| = 3$, $\gamma_{3}(G)\cong (C_{3})^{2}$. If $G'$ is any one of the groups  $S(729,104)$ or  $S(729,106)$, then $\gamma_{3}(G)\subseteq G'^{2}$. If $G'$ is  any one of the groups  $S(729,416)$, $S(729,419)$, $S(729,499)$ or $S(729,500)$, then $|G'^{3}\cap \gamma_{3}(G)| = 1$, $\gamma_{3}(G)\cong (C_{3})^{2}$ (see Table 6 of \cite{msbs3}).  For $p\geq 5$,  $|D_{(5),K}(G)| = 1$ leads to $G'^{p} = 1$,  $\gamma_{3}(G)\cong (C_{p})^{3}$, $G'' = \gamma_{4}(G)\cong C_{p}$, $|\zeta(G')|=  p^{3}$ or $p^{4}$.  If  $|\zeta(G')|=  p^{3}$, then no  group exists (see \cite{rj}). If  $|\zeta(G')|=  p^{4}$. Then $G'\cong \phi_{2}(1^5)\times (1)$, $\gamma_{3}(G)\cong (C_{p})^{3}$, $\zeta(G')\cong (C_{p})^{4}$ and $|G'^{p}\cap \gamma_{3}(G)| = 1$ (see \cite{rj}). \\
  		  
  		  Let $ \mathbf {d_{(4)} =  d_{(2)} = 1, d_{(3)} = 3}$. Then  $|G'| =  p^{5}$, $|D_{(4),K}(G)| = p$,  \sloppy   $|D_{(3),K}(G)| = p^{4}$. Let $G'$ be an abelian group. For $p= 2$,  $G'\cong (C_{4})^{2}\times C_{2}$ or $C_{4}\times (C_{2})^{3}$ or $(C_{2})^{4}$. If $G'\cong  (C_{4})^{2}\times C_{2}$,  then either $|G'^{2}\cap \gamma_{3}(G)| = 1$, $\gamma_{3}(G)\cong (C_{2})^{2}$ or $|G'^{2}\cap \gamma_{3}(G)| = 2$, $\gamma_{3}(G)\cong (C_{2})^{3} $ or $G'^{2}\subseteq \gamma_{3}(G)\cong (C_{2})^{4}$. If $G'\cong C_{4}\times (C_{2})^{3}$, then either $|G'^{2}\cap \gamma_{3}(G)| = 1$, $\gamma_{3}(G) \cong (C_{2})^{3}$ or $G'^{2}\subseteq \gamma_{3}(G)\cong (C_{2})^{4}$. If $G'\cong (C_{2})^{4}$, then $|G'^{2}\cap \gamma_{3}(G)| = 1$, $\gamma_{3}(G)\cong (C_{2})^{4}$. Now for $p = 3$, $|D_{(4),K}(G)| = 3$ leads to $|G'^{3}| = 3$. So $G'\cong C_{9}\times (C_{3})^{3}$, then either $\gamma_{3}(G)\cong (C_{3})^{3}$, $|G'^{3}\cap \gamma_{3}(G)| = 1$ or $\gamma_{3}(G)\cong (C_{3})^{4}$, $G'^{3}\subseteq \gamma_{3}(G)$. Now for $p\geq 5$, $D_{(5),K}(G) = 1$ leads to $G'^{p} = 1$. So  $G'\cong (C_{p})^{5}$, $\gamma_{3}(G)\cong (C_{p})^{4}$ and $|G'^{p}\cap \gamma_{3}(G)| = 1$. Let $G'$ be a non-abelian group. For $p= 2$,  $G'' = \gamma_{4}(G)\cong C_{2}$ and  $|\gamma_{3}(G)| \geq  4$. As $\gamma_{3}(G)\subseteq \zeta(G')$, hence $|\zeta(G')| = 4$ or 8. If $|\zeta(G')| = 4$, then from the Table 2 of \cite{rrm}, $|G''| \neq 2$. If $|\zeta(G')| =8$, then $\gamma_{3}(G) = \zeta(G')\cong (C_{2})^{3}$. Therefore only possible $G'$ is $S(32,2)$ and for this group $|G'^{2}\cap \gamma_{3}(G)| = 4$ (see Table 2 of \cite{rrm}).  For $p=  3$,  $G'' = \gamma_{4}(G)\cong C_{3}$, $G'^{3}\subseteq \gamma_{4}(G)\cong C_{3}$  and $|\gamma_{3}(G)| \geq 3^{2}$.  If $|\gamma_{3}(G)| = 3^{2} $ and  $|\zeta(G')| = 3^2$ or $3^3$, then  from the Table 5 of \cite{msbs3} no such group exists. If $|\gamma_{3}(G)| = 3^{3}$, then $\gamma_{3}(G) = \zeta(G') \cong (C_{3})^{3}$, so only possible $G'$ is $S(243,32)$ and for this group $|G'^{3}\cap \gamma_{3}(G)| = 1$. For $p\geq 5$, $D_{(5),K}(G) = G'^{p}\gamma_{3}(G)^{p}\gamma_{5}(G) = 1$ leads to $G'^{p} = 1$. Now $ G'' = \gamma_{4}(G)\cong C_{p}$,   $\gamma_{3}(G)\cong (C_{p})^{4}$  and $\gamma_{3}(G)=  \zeta(G')\cong (C_{p})^{4}$. Thus $G'$ is abelian in this case.\\
  		  
  		  Let $\mathbf  {d_{(4)} = d_{(3)} = 2}$. Since $d_{(1+1)} = 0$, therefore by Lemma \ref{2l1}(2), $\vartheta_{p'}(2)\geq \vartheta_{p'}(1)$ for all $p>0$ and so $d_{(3)} = 0$.\\

  		  Let $\mathbf  {d_{(4)} = 0}$.  Then  we have the following possibilities:
  		  $d_{(2)} = 10$ or $d_{(2)} = 8$, $d_{(3)} = 1$ or $d_{(2)} = 6$, $d_{(3)} = 2$ or $d_{(2)} = 4$, $d_{(3)} = 3$ or $d_{(2)} = 2$, $d_{(3)} = 4$ or  $d_{(3)} = 5$.\\
  		  
  		 Let  $\mathbf  {d_{(2)} = 10}$.  Then  $|G'| = p^{10}$, $|D_{(3),K}(G)| = 1$ and hence  $G'^{p} = \gamma_{3}(G) = 1$, for all $p>0$.  Thus  $G'$ is abelian and  $G'\cong (C_{p})^{10}$, $\gamma_{3}(G)= 1$. \\
  		 
  		 Let  $\mathbf  {d_{(2)} = 8, d_{(3)} = 1}$. Thus  $|G'|= p^{9}$, $|D_{(3),K}(G)| = p$ and $G'$ is abelian for all $p>0$.  For $p\geq 3$,  $G'^{p} =  1$  and hence  $G'\cong (C_{p})^{9}$,  $\gamma_{3}(G)\cong C_{p}$, $|G'^{p}\cap \gamma_{3}(G)| = 1$. For $p =2$, $|D_{(3),K}(G)| = 2$ leads to $|G'^{2}|\leq 2$. So $G'\cong C_{4}\times (C_{2})^{7}$ or $(C_{2})^{9}$. If $G'\cong C_{4}\times (C_{2})^{7}$, then $\gamma_{3}(G)\subseteq G'^{2}\cong C_{2}$. If $G'\cong (C_{2})^{9}$, then $\gamma_{3}(G)\cong C_{2}$, $|G'^{2}\cap \gamma_{3}(G)| = 1$.\\
  		 
  		 Let $\mathbf {d_{(2)} = 6, d_{(3)} = 2}$. Thus $|G'|= p^{8}$, $|D_{(3),K}(G)| = p^{2}$ and $G'$ is abelian for  all $p>0$ . For $p\geq 3$, $G'^{p} =  1$, hence $G'\cong (C_{p})^{8}$,  $\gamma_{3}(G)\cong (C_{p})^{2} $ and $|G'^{p}\cap \gamma_{3}(G)| = 1$. For $p =2$, $|D_{(3),K}(G)| = 2^{2}$ leads to $|G'^{2}|\leq 4$. So $G'\cong (C_{2})^{8}$ or $C_{4}\times (C_{2})^{6}$ or $(C_{4})^{2}\times (C_{2})^{4}$. If $G'\cong (C_{2})^{8}$, then $|G'^{2}\cap \gamma_{3}(G)| = 1$, $\gamma_{3}(G)\cong (C_{2})^{2}$. If $G'\cong C_{4}\times (C_{2})^{6}$, then  either $|G'^{2}\cap \gamma_{3}(G)| = 1$, $\gamma_{3}(G)\cong C_{2}$ or $G'^{2}\subseteq \gamma_{3}(G)\cong (C_{2})^{2}$. If $G'\cong (C_{4})^{2}\times (C_{2})^{4}$, then $\gamma_{3}(G)\subseteq G'^{2}$. \\
  		 
  		 Let  $\mathbf  {d_{(2)} = 4, d_{(3)} = 3} $.  Thus $|G'|= p^{7}$, $|D_{(3),K}(G)| = p^{3}$ and $G'$ is abelian, for all $p>0$. If $p\geq 3$, then $G'\cong (C_{p})^{7}$ and $\gamma_{3}(G)\cong (C_{p})^{3}$, $|G'^{p}\cap \gamma_{3}(G)| = 1$. For $p =2$, $|D_{(3),K}(G)| = 2^{3}$ leads to $|G'^{2}|\leq 8$. So $G'\cong (C_{4})^{3}\times C_{2}$ or $ (C_{4})^{2}\times (C_{2})^{3}$ or $C_{4}\times (C_{2})^{5}$ or $(C_{2})^{7}$. If $G'\cong (C_{4})^{3}\times C_{2}$, then $\gamma_{3}(G)\subseteq G'^{2}$. If $G'\cong  (C_{4})^{2}\times (C_{2})^{3}$, then either $|G'^{2}\cap \gamma_{3}(G)| = 1$, $\gamma_{3}(G)\cong C_{2}$ or $|G'^{2}\cap \gamma_{3}(G)| = 2$, $\gamma_{3}(G)\cong C_{2}\times C_{2}$ or $G'^{2}\subseteq \gamma_{3}(G)\cong (C_{2})^{3}$. If $G'\cong C_{4}\times (C_{2})^{5}$, then either $|G'^{2}\cap \gamma_{3}(G)| = 1$, $\gamma_{3}(G)\cong (C_{2})^{2}$ or $G'^{2}\subseteq \gamma_{3}(G)\cong (C_{2})^{3}$. If $G'\cong (C_{2})^{7}$, then $|G'^{2}\cap \gamma_{3}(G)| = 1$, $\gamma_{3}(G)\cong (C_{2})^{3}$. \\
  		 
  		 Let $\mathbf  {d_{(2)} = 2, d_{(3)} = 4}$. Thus  $|G'| =p^{6}$, $|D_{(3),K}(G)| = p^{4}$  and $G'$ is abelian for all $p>0$. For $p\geq 3$,  $G'^{p} = 1$, so $G'\cong (C_{p})^{6}$, $|G'^{p}\cap \gamma_{3}(G)| = 1$ and $\gamma_{3}(G)\cong (C_{p})^{4}$. For $p = 2$, $|D_{(3),K}(G)| = 2^{4}$ leads to $|G'^{2}| \leq 8$. So $G'\cong (C_{4})^{3}$ or $(C_{4})^{2}\times (C_{2})^{2}$ or $C_{4}\times (C_{2})^{4}$ or $(C_{2})^{6}$. If $G'\cong (C_{4})^{3}$, then $\gamma_{3}(G)\subseteq G'^{2}$. If $G'\cong (C_{4})^{2}\times (C_{2})^{2}$, then  either $|G'^{2}\cap \gamma_{3}(G)| = 1$, $\gamma_{3}(G)\cong (C_{2})^{2}$ or $|G'^{2}\cap \gamma_{3}(G)| = 2$, $\gamma_{3}(G)\cong (C_{2})^{3}$ or $G'^{2}\subseteq \gamma_{3}(G)\cong (C_{2})^{4}$. If $G'\cong C_{4}\times (C_{2})^{4}$, then either $|G'^{2}\cap \gamma_{3}(G)| = 1$, $\gamma_{3}(G)\cong (C_{2})^{3}$ or $G'^{2}\subseteq \gamma_{3}(G) \cong (C_{2})^{4}$. If $G'\cong (C_{2})^{6}$, then $|G'^{2}\cap \gamma_{3}(G)| = 1$, $\gamma_{3}(G)\cong (C_{2})^{4}$.  \\
  		 
  		 Let  $\mathbf  {d_{(3)} = 5}$. Since $d_{(1+1)} = 0$, therefore by Lemma \ref{2l1}(2), $\vartheta_{p'}(2)\geq \vartheta_{p'}(1)$ for all $p>0$ and so $d_{(3)} = 0$. \\
  		 
  		Converse can be easily done by computing $d_{(m)}$'s in each case.
  		  
  		\end{proof}
  		\section*{Acknowledgment}
  		The second author is thankful to Science and Engineering Research board (SERB), India [grant number: ECR/2016/001274] for financial assistance.

  \begin{center}
  	
  	\begin{table}[ht]
  			
  		{\begin{tabular}{@{}cccccccc@{}} \toprule
  					$G'$ & $G'^{5}$  & exp($G'$)  &$\zeta(G')$ & $G''$  & $G''\cap G'^{5}$ &$G''\cap \zeta(G')$ & $G'^{5}\cap \zeta(G')$ \\ 
  					\hline\\
  					S(3125,2) &$C_{5}\times C_{5}$ & 25 & $C_{5}\times C_{5}\times C_{5}$& $C_{5}$ &1& $C_{5}$&$C_{5}\times C_{5}$ \\ [.5ex]
  					S(3125,16) &$C_{25}\times C_{5}$ & 125 & $C_{25}\times C_{5}$& $C_{5}$&$C_{5}$& $C_{5}$&$C_{25}\times C_{5}$  \\ [.5ex]
  					S(3125,17) & $C_{25}$& 125 & $C_{25}\times C_{5}$& $C_{5}$ & 1& $C_{5}$&$C_{25}$  \\ [.5ex]
  					S(3125,26) &$C_{25}\times C_{5}$ & 125& $C_{25}\times C_{5}$&$ C_{5}$ &$ C_{5}$& $ C_{5}$&$C_{25}\times C_{5}$ \\ [.5ex]

  					S(3125,29) &$C_{125}$ & 625& $C_{125}$&$ C_{5}$ &$ C_{5}$& $ C_{5}$&$C_{125}$ \\ [.5ex]
  					S(3125,40) &$ C_{5}$ & 25& $C_{5}\times C_{5}\times C_{5}$&$ C_{5}$ &1& $ C_{5}$&$ C_{5}$ \\ [.5ex]
  					S(3125,41) &$C_{5}\times C_{5}$ & 25& $C_{5}\times C_{5}\times C_{5}$&$ C_{5}$ &$ C_{5}$& $ C_{5}$&$C_{5}\times C_{5}$ \\ [.5ex]
  					S(3125,42) &$C_{5}\times C_{5}$ & 25& $C_{25}\times C_{5}$&$ C_{5}$ &$ C_{5}$& $ C_{5}$&$C_{5}\times C_{5}$ \\ [.5ex]
  					S(3125,43) &$ C_{5}$ & 25& $C_{25}\times C_{5}$&$ C_{5}$ &1& $ C_{5}$&$ C_{5}$ \\ [.5ex]
  					S(3125,44) &$ C_{5}\times C_{5}$ & 25& $C_{25}\times C_{5}$&$ C_{5}$ &$C_{5}$& $ C_{5}$&$ C_{5}\times C_{5}$\\ [.5ex]
  					S(3125,59)&$ C_{25}$ & 125& $C_{25}\times C_{5}$&$ C_{5}$ &$C_{5}$& $ C_{5}$&$  C_{25}$ \\ [.5ex]

  					S(3125,60) &$ C_{25}$ & 125& $C_{125}$&$ C_{5}$ &$C_{5}$& $ C_{5}$&$ C_{25} $ \\ [.5ex]
  					S(3125,72) &1 & 5& $C_{5}\times C_{5}\times C_{5}$&$ C_{5}$ &1& $ C_{5}$&1  \\ [.5ex]
  					S(3125,73)  &$C_{5}$ & 25& $C_{5}\times C_{5}\times C_{5}$&$ C_{5}$ &$ C_{5}$& $ C_{5}$&$C_{5}$ \\ [.5ex]

  					S(3125,74)  &$C_{5}$ & 25& $C_{25}\times C_{5}$&$ C_{5}$ &$ C_{5}$& $ C_{5}$&$C_{5}$  \\ [.5ex]
  					S(3125,75) &1 & 5& $ C_{5}$&$ C_{5}$ &1& $ C_{5}$&1 \\ [.5ex]
  					S(3125,76) &$C_{5}$ & 25& $C_{5}$&$ C_{5}$ &$ C_{5}$& $ C_{5}$&$ C_{5}$

  				\label{64}
  			\end{tabular}}
  		
  		\end{table}

  \end{center}
  \begin{center}

  \begin{table}[ht]
  
  	{\begin{tabular}{@{}ccccccc@{}} \toprule

  		$G'$ & $G'^{3}$  & exp($G'$)  &$\zeta(G')$ & $G''$  & $G''\cap G'^{3}$ & $G'^{3}\cap \zeta(G')$ \\ 
  			\hline\\
  	
  		S(2187,5867) &$C_{3}\times C_{3}$ &  9 & $C_{9}\times C_{9}\times C_{3}$& $C_{3}$ & $1$ &$C_{3}\times C_{3}$ \\ [.5ex]
  		S(2187,5868) &$C_{3}\times C_{3}\times C_{3}$ & 9 & $C_{9}\times C_{9}\times C_{3}$ & $C_{3}$&$C_{3}$& $C_{3}\times C_{3}\times C_{3} $  \\ [.5ex]
  		S(2187,5869) &$C_{3}\times C_{3}\times C_{3}$ & 9 & $C_{9}\times C_{9}\times C_{3}$ & $C_{3}$&$C_{3}$& $C_{3}\times C_{3}\times C_{3} $  \\ [.5ex]
  		S(2187,5870) &$C_{3}\times C_{3}$ & 9& $C_{9}\times C_{3}\times C_{3}\times C_{3}$&$ C_{3}$& $1$  &$C_{3}\times C_{3}$ \\[.5ex]

  		S(2187,5871) &$C_{3}\times C_{3}\times C_{3}$ & 9& $C_{9}\times C_{3}\times C_{3}\times C_{3}$&$ C_{3}$ &$ C_{3}$&  $C_{3}\times C_{3}\times C_{3}$ \\ [.5ex]
  		S(2187,5872) &$ C_{3}\times C_{3}$ & 9& $C_{3}\times C_{3}\times C_{3}\times C_{3}\times C_{3}$&$ C_{3}$ &1 & $C_{3}\times C_{3}$ \\ [.5ex]
  		S(2187,5873) &$C_{3}\times C_{3} \times C_{3}$ & 9& $C_{9}\times C_{3}\times C_{3}\times C_{3}$&$ C_{3}$ &$ C_{3}$&$C_{3}\times C_{3}\times C_{3}$ \\ [.5ex]
  		S(2187,5874) &$C_{3}\times C_{3}$ & 9& $C_{3}\times C_{3}\times C_{3}$&$ C_{3}$ & $1$&$C_{3}\times C_{3}$ \\ [.5ex]
  		S(2187,5875) &$ C_{3}\times C_{3}\times C_{3}$ & 9& $C_{3}\times C_{3}\times C_{3}$&$ C_{3}$ & $ C_{3}$&$ C_{3}\times C_{3}\times C_{3} $ \\ [.5ex]
  		S(2187,5876) &$ C_{3}\times C_{3}$ & 9& $C_{3}\times C_{3}\times C_{3}$&$ C_{3}$ &$1$& $C_{3}\times C_{3}$\\ [.5ex]
  		S(2187,5877)&$ C_{3}\times C_{3}\times C_{3}$ & 9& $C_{3}\times C_{3}\times C_{3}$&$ C_{3}$ &$C_{3}$&$C_{3}\times C_{3}\times C_{3}$ \\ [.5ex]

  		S(2187,9094) &$ C_{3}$ & 9& $C_{3}\times C_{3}\times C_{3}\times C_{3}\times C_{3}$&$ C_{3}$ &$1$&  $ C_{3}$ \\ [.5ex]
  		S(2187,9095) &$ C_{3}$ & 9& $C_{9}\times C_{3}\times C_{3}\times C_{3}$&$ C_{3}$ &$1$&  $ C_{3}$ \\ [.5ex]
  	
  		S(2187,9096)  &$C_{3}\times C_{3}$ & 9& $C_{9}\times C_{9}\times C_{3}$&$ C_{3}$ &$ C_{3}$ &$C_{3}\times C_{3}$ \\ [.5ex]

  		S(2187,9097)  &$C_{3}\times C_{3}$ & 9& $C_{9}\times C_{3}\times C_{3}\times C_{3} $&$ C_{3}$ &$ C_{3}$ &$C_{3}\times C_{3}$ \\ [.5ex]
  	 	S(2187,9098) &$C_{3}\times C_{3}$ & 9& $C_{9}\times C_{3}\times C_{3}\times C_{3} $&$ C_{3}$ &$ C_{3}$ &$C_{3}\times C_{3}$ \\ [.5ex]
  	 	S(2187,9099)  &$C_{3}\times C_{3}$ & 9& $C_{3}\times C_{3}\times C_{3}\times C_{3}\times C_{3}$&$ C_{3}$ &$ C_{3}$& $ C_{3}\times C_{3}$  \\ [.5ex]
  	 	S(2187,9100)  &$C_{3}$ & 9& $C_{3}\times C_{3}\times C_{3}$&$ C_{3}$ &$ 1$& $ C_{3}$  \\ [.5ex]
  	 	S(2187,9101)  &$C_{3}$ & 9& $C_{9}\times C_{3}$&$ C_{3}$ &$ 1$& $ C_{3}$  \\ [.5ex]
  	 	S(2187,9102)  &$C_{3}\times C_{3}$ & 9& $C_{9}\times C_{3}$&$ C_{3}$ &$ C_{3}$& $ C_{3}\times C_{3}$  \\ [.5ex]
  	 	S(2187,9103)  &$C_{3}\times C_{3}$ & 9& $C_{9}\times C_{3}$&$ C_{3}$ &$ C_{3}$& $ C_{3}\times C_{3}$  \\ [.5ex]
  	 	S(2187,9104)  &$C_{3}\times C_{3}$ & 9& $C_{3}\times C_{3}\times C_{3}$&$ C_{3}$ &$ C_{3}$& $ C_{3}\times C_{3}$  \\ [.5ex]
  	 	S(2187,9105)  &$C_{3}\times C_{3}$ & 9& $C_{3}\times C_{3}\times C_{3}$&$ C_{3}$ &$ C_{3}$& $ C_{3}\times C_{3}$  \\ [.5ex]
  	 	S(2187,9303)  &$C_{3}$ & 9& $C_{3}\times C_{3}\times C_{3}\times C_{3}\times C_{3}$&$ C_{3}$ &$ C_{3}$& $ C_{3}$  \\ [.5ex]
  	 	S(2187,9304)  &$C_{3}$ & 9& $C_{9}\times C_{3}\times C_{3}\times C_{3} $&$ C_{3}$ &$ C_{3}$& $ C_{3}$  \\ [.5ex]
  	 	
  	 	S(2187,9306)  &$C_{3}$ & 9& $C_{3}\times C_{3}\times C_{3}$&$ C_{3}$& $ C_{3}$&$C_{3}$  \\ [.5ex]
  	 	S(2187,9307)  &$C_{3}$ & 9& $C_{9}\times C_{3}$&$ C_{3}$ &$ C_{3}$& $ C_{3}$  \\ [.5ex]
  		S(2187,9309)  &$C_{3}$ & 9& $C_{3}$&$ C_{3}$ &$ C_{3}$& $ C_{3}$ 
  		\label{2187}
  \end{tabular}}
  
\end{table}

\end{center}

\begin{center}
	
	 \begin{table}[ht]
	 	
	 	{\begin{tabular}{@{}ccccccc@{}} \toprule

$G'$ & $G'^{3}$  & exp($G'$)  &$\zeta(G')$ & $G''$  & $G''\cap G'^{3}$ & $G'^{3}\cap \zeta(G')$ \\ 
	\hline\\

	S(243,13) &$C_{3}\times C_{3}$ & 9 & $C_{3}\times C_{3}$& $C_{3}\times C_{3}$ & $C_{3}$ &$C_{3}\times C_{3}$ \\ [.5ex]
	S(243,14) &$C_{3}\times C_{3}$ & 9 & $C_{3}\times C_{3}$& $C_{3}\times C_{3}$ & $C_{3}$ &$C_{3}\times C_{3}$  \\ [.5ex]
	S(243,15) &$C_{3}\times C_{3}$ & 9 & $C_{3}\times C_{3}$& $C_{3}\times C_{3}$ & $C_{3}$ &$C_{3}\times C_{3}$ \\ [.5ex]
	
	S(243,16) &$C_{9}$ & 27& $C_{9}$&$ C_{3}\times C_{3}$& $C_{3}$  &$C_{9}$  \\ [.5ex]
	S(243,17) &$C_{3}\times C_{3}$ & 9 & $C_{3}\times C_{3}$& $C_{3}\times C_{3}$ & $C_{3}$ &$C_{3}\times C_{3}$ \\ [.5ex]
	S(243,18) &$C_{3}\times C_{3}$ & 9 & $C_{3}\times C_{3}$& $C_{3}\times C_{3}$ & $C_{3}$ &$C_{3}\times C_{3}$ \\ [.5ex]
	S(243,19) &$C_{9}$ & 27& $C_{9}$&$ C_{3}\times C_{3}$& $C_{3}$  &$C_{9}$ \\ [.5ex]
	S(243,20) &$C_{9}$ & 27& $C_{9}$&$ C_{3}\times C_{3}$& $C_{3}$  &$C_{9}$ \\ [.5ex]
	S(243,22) &$C_{9}\times C_{3}$ & 27& $C_{3}$&$ C_{9}$& $C_{9}$  &$C_{3}$ \\ [.5ex]
	S(243,37) &$1$ & 3 & $C_{3}\times C_{3}$& $C_{3}\times C_{3}$ & $1$ &$1$ \\ [.5ex]
	S(243,38) &$C_{3}$ & 9 & $C_{3}\times C_{3}$& $C_{3}\times C_{3}$ & $C_{3}$ &$C_{3}$  \\ [.5ex]
	S(243,39) &$C_{3}$ & 9 & $C_{3}\times C_{3}$& $C_{3}\times C_{3}$ & $C_{3}$ &$C_{3}$ \\ [.5ex]
	
	S(243,40) &$C_{3}$ & 9& $C_{3}\times C_{3}$&$ C_{3}\times C_{3}$& $C_{3}$  &$C_{3}$  \\ [.5ex]
	S(243,41) &$C_{3}\times C_{3}$ & 9 & $C_{3}\times C_{3}$& $C_{3}\times C_{3}$ & $C_{3}\times C_{3}$ &$C_{3}\times C_{3}$ \\ [.5ex]
	S(243,42) &$C_{3}\times C_{3}$ & 9 & $C_{3}\times C_{3}$& $C_{3}\times C_{3}$ & $C_{3}\times C_{3}$ &$C_{3}\times C_{3}$ \\ [.5ex]
	S(243,43) &$C_{3}\times C_{3}$ & 9& $C_{3}\times C_{3}$&$ C_{3}\times C_{3}$& $C_{3}\times C_{3}$  &$C_{3}\times C_{3}$ \\ [.5ex]
	S(243,44) &$C_{3}\times C_{3}$ & 9& $C_{3}\times C_{3}$&$ C_{3}\times C_{3}$& $C_{3}\times C_{3}$  &$C_{3}\times C_{3}$ \\ [.5ex]
	S(243,45) &$C_{3}\times C_{3}$ & 9& $C_{3}\times C_{3}$&$ C_{3}\times C_{3}$& $C_{3}\times C_{3}$  &$C_{3}\times C_{3}$ \\ [.5ex]
	S(243,46) &$C_{3}\times C_{3}$ & 9 & $C_{3}\times C_{3}$& $C_{3}\times C_{3}$ & $C_{3}\times C_{3}$ &$C_{3}\times C_{3}$ \\ [.5ex]
	S(243,47) &$C_{3}\times C_{3}$ & 9 & $C_{3}\times C_{3}$& $C_{3}\times C_{3}$ & $C_{3}\times C_{3}$ &$C_{3}\times C_{3}$  \\ [.5ex]
	S(243,51) &$C_{3}$ & 9 & $C_{3}\times C_{3}$& $C_{3}\times C_{3}$ & $C_{3}$ &$C_{3}$ \\ [.5ex]
	
	S(243,52) &$C_{3}$ & 9& $C_{3}\times C_{3}$&$ C_{3}\times C_{3}$& $C_{3}$  &$C_{3}$  \\ [.5ex]
	S(243,53) &$C_{3}$ & 9 & $C_{3}\times C_{3}$& $C_{3}\times C_{3}$ & $C_{3}$ &$C_{3}$ \\ [.5ex]
	S(243,54) &$C_{3}$ & 9 & $C_{3}\times C_{3}$& $C_{3}\times C_{3}$ & $C_{3}$ &$C_{3}$ \\ [.5ex]
	S(243,55) &$C_{3}$ & 9& $C_{9}$&$ C_{3}\times C_{3}$& $C_{3}$  &$C_{3}$ \\ [.5ex]
	S(243,56) &$C_{3}$ & 9& $C_{3}$&$ C_{3}\times C_{3}$& $C_{3}$  &$C_{3}$ \\ [.5ex]
	S(243,57) &$C_{3}$ & 9& $C_{3}$&$ C_{3}\times C_{3}$& $C_{3}$  &$C_{3}$ \\ [.5ex]
	S(243,58) &$C_{3}$ & 9 & $C_{3}$& $C_{3}\times C_{3}$ & $C_{3}$ &$C_{3}$ \\ [.5ex]
	S(243,59) &$C_{3}$ & 9& $C_{3}$&$ C_{3}\times C_{3}$& $C_{3}$  &$C_{3}$ \\ [.5ex]
	S(243,60) &$C_{3}$ & 9& $C_{3}$&$ C_{3}\times C_{3}$& $C_{3}$  &$C_{3}$ 
		\label{243}
\end{tabular}}

\end{table}

\end{center}
    \end{document}